\documentclass[11pt,reqno]{amsart}
\usepackage[a4paper,margin=21mm]{geometry} % the page geometry
\usepackage{mathtools,amssymb,bm,amsthm} % the standard math libraries (mathtools => amsmath)
\usepackage{graphicx}
\usepackage{tikz-cd}
\usepackage[e]{esvect} % used for the definition of \vv
\usepackage{breakcites} % to allow line breakin in the citations
\usepackage{hyperref}
\usepackage{calrsfs}
\hypersetup{colorlinks=true, allcolors=blue!50!black}
\numberwithin{equation}{section}
\theoremstyle{plain}

\newtheorem{theorem}{Theorem}[section]
\newtheorem{corollary}[theorem]{Corollary}
\newtheorem{lemma}[theorem]{Lemma}
\newtheorem{proposition}[theorem]{Proposition}
\theoremstyle{definition}
\newtheorem{definition}[theorem]{Definition}
\newtheorem{remark}[theorem]{Remark}
\DeclareMathAlphabet{\pazocal}{OMS}{zplm}{m}{n}
\let\eps\varepsilon
\newcommand{\R}{\mathbb{R}}% the real numbers
\let\S\relax% free the standard \S to replace it
\newcommand{\S}{\mathbb{S}}% the sphere
\newcommand*{\hausdorff}[1]{\mathcal{H}^{#1}}% the Hausdorff measure of dimension #1
\newcommand*{\comp}[1]{{#1}^{c}}% the complementary of a set
\newcommand*{\supp}[1]{\operatorname{supp}(#1)}% the support of a measure
\newcommand*{\inter}[1]{\operatorname{int}(#1)}% the interior of a set
\newcommand*{\Star}[1]{\operatorname{Star}(#1)}% all the hyperplanes containing the line
\newcommand*{\cyl}[1]{\operatorname{Cylpair}(#1)}% a cylindrical pair
\newcommand*{\Gr}[2][n-1]{\operatorname{Gr}_{#1}(#2)}% the grassmanian
\newcommand*{\Strcvx}[2][n-1]{\operatorname{Strcvx}_{#1}(#2)}% a neighborhood base of \Star
\newcommand*{\regpts}[1]{\operatorname{reg}{#1}}% the set of regular points on the boundary
\newcommand*{\vect}[1]{\operatorname{span}(#1)}% vector span (\span is already defined)
\newcommand{\area}{\mathcal{A}}% the Holmes-Thompson area
\DeclareMathOperator{\vol}{\pazocal{V}}% the volume
\DeclareMathOperator{\LL}{\pazocal{L}}% the Legendre transform
\DeclareMathOperator{\CC}{\pazocal{C}}% the dual Santaló point
\DeclareMathOperator{\CS}{\mathcal{C}}% the dual Santaló point in the Euclidean case
\let\SS\relax % free the standard \SS to replace it
\DeclareMathOperator{\SS}{\pazocal{S}}% the Santaló point

\def\B{B}% the regular convex body (generalizing the ball)
\def\K{K}% an arbitrary convex body
\def\D{B}% same as \B or \K* in section 2
\def\C{K}% same as \K or \B* in section 2

\usepackage{xcolor}

\title[Holmes-Thompson area \& Santal\'o point]{The Santal\'o point for the Holmes-Thompson boundary area}
\author{Florent Balacheff}
\address{Florent Balacheff, Departament de Matem\`atiques, Universitat Aut\`onoma de Barcelona, Barcelona, Spain}
\email{fbalacheff@mat.uab.cat}
\author{Gil Solanes}
\address{Gil Solanes, Departament de Matem\`atiques, Universitat Aut\`onoma de Barcelona, Barcelona, Spain and Centre de Recerca Matem\`atica, Barcelona, Spain.}
\email{solanes@mat.uab.cat}
\author{Kroum Tzanev}
\address{Kroum Tzanev, Laboratoire Paul Painlev\'e, Universit\'e de Lille, Villeneuve d'Ascq, France.}
\email{kroum.tzanev@univ-lille.fr}
\thanks{The first author acknowledges support by the FSE/AEI/MICINN grant RYC-2016-19334  ``Local and global systolic geometry and topology". The second author is supported by the Serra Hunter Programme.  The first and the second authors acknowledge support by the FEDER/AEI/MICIU grant PGC2018-095998-B-I00 ``Local and global invariants in geometry". }
\begin{document}
%%%%%%%%%%%%%%%%%%%%%%%%
\begin{abstract}
  We explore the notion of Santal\'o point for the Holmes-Thompson boundary area of a convex body in a normed space. In the case where the norm is $C^1$, and in the case where unit ball and convex body coincide,  we prove existence and uniqueness. When the normed space has a smooth positively curved unit ball, we exhibit a dual Santal\'o point expressed as an average of centroids of projections of the dual body.
\end{abstract}
%%%%%%%%%%%%%%%%%%%%%%%%
\maketitle

%%%%%%%%%%%%%%%%%%%%%%%%
\section{introduction}
%%%%%%%%%%%%%%%%%%%%%%%%

Several decades ago Santal\'o studied in \cite{San49} the functional
\begin{equation}\label{eq:vol_dual}
  x \in \inter{\K} \mapsto |(\K-x)^\circ|
\end{equation}
associated to any convex body $\K$ of $\R^n$. Here $|\cdot |$ denotes the Lebesgue measure of $\R^n$, and $A^\circ=\{y \in {\R}^n \mid \langle x,y\rangle=\sum_{i=1}^n x_iy_i\leq 1 \, \, \forall x \in A\}$ is the polar set of a convex body $A$. He found that this functional is proper and strictly convex. Consequently, there exists a unique minimizing point $s(\K)$ in the interior of $\K$, known nowadays as the {\it Santal\'o point} of $\K$. Along the way, Santal\'o computed the derivative of \eqref{eq:vol_dual} and showed that
\begin{equation}\label{eq:var_vol_dual}
  \left. \frac{d}{dt}\right|_{t=0} |(\K-tv)^\circ|=(n+1)\left\langle c(\K^\circ),v\right\rangle
\end{equation}
where $c(\K^\circ)=\int_{\K^\circ}  x \, \, dx$ is the centroid of the polar body $\K^\circ$ and $v$ any vector in $\R^n$. Therefore $s(\K)$ lies at the origin if and only if $c(\K^\circ)$ lies at the origin. Determining the Santal\'o point of a given convex body can be a difficult question, so this characterization is particularly useful.

As the Lebesgue measure coincides with the $n$-dimensional Hausdorff measure of $\R^n$, it is natural to consider the following functional similar to (\ref{eq:vol_dual})
\begin{equation}\label{eq:area_dual}
  x \in \inter{\K} \mapsto \hausdorff{n-1}(\partial(\K-x)^\circ),
\end{equation}
where $\hausdorff{n-1}$ denotes the $(n-1)$-dimensional Hausdorff measure, and look for possible minimizing points.

More generally, let $V$ be an  $n$-dimensional real vector space. To any convex body $\B$ in $V$ that contains the origin in its interior is associated a unique (asymmetric) norm $\| \cdot \|_{\B}$ whose unit ball is precisely $\B$. The restriction of $\|\cdot \|_{\B}$ to any hypersurface defines a Finsler metric whose corresponding Holmes-Thompson $(n-1)$-volume (or area in short)  will be denoted by $\area_{\B}(\cdot)$. If $\K$ is a convex body in $(V,\|\cdot\|_{\B})$, we can consider its dual body $\K^\ast=\{p \in V^\ast \mid p(x)\leq 1 \, \, \forall x \in \K\}$ as a convex body of the dual normed vector space $(V^\ast,\|\cdot\|_{\B^\ast})$ and focus on the associated Holmes-Thompson area of its boundary sphere, that is the quantity $\area_{\B^\ast}(\partial \K^\ast)$. In the special case where $V=\R^n$ and $\B$ is the Euclidean unit ball, we have $\area_{\B^\ast}(\partial \K^\ast)=\hausdorff{n-1}(\partial \K^\circ).$ Therefore functional  (\ref{eq:area_dual}) turns out to be a particular case of the functional
\[
  x \in \inter{\K}\mapsto \area_{\B^\ast}(\partial (\K-x)^\ast).
\]
A minimizing point of this functional will be called a {\it Santal\'o point} of $\K$ for the Holmes-Thompson area in $(V,\|\cdot\|_{\B})$.

In this paper we first prove existence and uniqueness of Santal\'o points for the Holmes-Thompson area in any normed vector space whose unit ball is of class $C^1$.

\begin{theorem}\label{th:santalo}
  Let $\K$ be a convex body in a finite-dimensional real normed vector space $(V,\|\cdot\|_{\B})$ whose unit ball $\B$ is of class $C^1$. The functional
  \[
    x \in \inter{\K}\mapsto \area_{\B^\ast}(\partial (\K-x)^\ast)
  \]
  is strictly convex and proper on the interior of $\K$.

  In particular there exists a unique minimizing point $\SS_{\B}(\K) \in \inter{\K}$ for this functional.
\end{theorem}

In case the unit ball is no longer of class $C^1$, the above functional is still convex  (and proper)  on the interior of $\K$, but strict convexity is no longer guaranteed. Therefore,  Santal\'o points do exist but are no longer necessarily unique. In section \ref{sec:convexity}, our study goes far beyond the $C^1$-regularity, and we give precise necessary and sufficient conditions for the uniqueness. For example, when the unit ball $B$ is a polytope there is always a convex body $K$ with an infinite set of Santal\'o points, as illustrated in the concrete example \ref{rem:cex}. On the bright side, uniqueness is always ensured without any regularity assumption in the special case where unit ball and convex body coincide. Using the duality formula $\area_{\B^\ast}(\partial \K^\ast)=\area_{\K}(\partial \B)$ discovered by \cite{HT79}, we restate this result as follows.

\begin{theorem}\label{th:santaloK=B}
Let $(V,\|\cdot\|_{\B})$ be any finite-dimensional real normed vector space.
The functional 
\[
	x\in \inter{\D} \mapsto \area_{\D-x}(\partial (\D-x))
\]
admits a unique minimizing point $\SS_{\D}(\D)$.
\end{theorem}

Rephrasing the theorem above, there exists a unique translation of the unit ball of a given normed vector space minimizing its own Holmes-Thompson boundary area.

In addition, by invariance properties of the Holmes-Thompson area, the map $\D\to \SS_{\D}(\D)$ is affinely equivariant, that is 
\[
\SS_{T\D}(T\D)=T(\SS_{\D}(\D))
\] 
for any invertible affine map $T$. Therefore this map defines a new affine-invariant point in the sense of \cite{Grun63}. See subsection \ref{sec:affineinvariantpoint} for more details.

In order to characterize $\SS_{\B}(\K)$, we study the first variation of the above functional. Assuming that the normed vector space $(V,\|\cdot\|_{\B})$ is {\it Minkowski}, i.e. $\partial\B$ is smooth and positively curved  (that is, with strictly positive sectional curvature), we obtain an interesting formula closely ressembling \eqref{eq:var_vol_dual} and involving some classical  notions from affine differential geometry.

Given $x\in\partial \B$, recall that the affine normal line is a canonical $1$-dimensional vector subspace $N_x\subset V$ transverse to $T_x=T_x\partial \B$, which behaves equivariantly under linear transformations of $\partial \B$. We thus have the decompositions
\[
  V=T_x\oplus N_x \qquad \text{and} \qquad V^*=N_x^\bot\oplus T_x^\bot.
\]
Let us consider the isomorphism $N_x^\bot\simeq T_x^*$ given by $\phi\mapsto\phi|_{T_x}$, and the projection $\pi_x\colon V^*\to N_x^\bot$ with kernel $\ker\pi_x=T_x^\bot$. Pick an arbitrary smooth measure $\mu$ on $\partial \B$ and consider the Lebesgue measure $\mu_x$ on $T_x$, and its dual measure $(\mu_x)^*$ on $T_x^*$. Let $\nu_x$ be the Lebesgue measure on $N_x^\bot\simeq T_x^*$ corresponding to $(\mu_x)^*$.

\begin{definition}
  Let $\K$ be a convex body in a Minkowski space $(V,\|\cdot\|_{\B})$. Using the previous notation, we define
  \[
   \CC_{\B}(\K^\ast)=\int_{\partial \B}\left(\int_{\pi_x (\K^\ast)} \, \, q \, d\nu_x(q)\right)\,  d\mu(x).
  \]
\end{definition}

Note that the inner integral is the centroid of the projection $\pi_x(\K^\ast)$ with respect to $\nu_x$ and belongs to $V^\ast$. It is easy to see that $\CC_{\B}(\K^\ast)$ does not depend on the choice of $\mu$.

Our third main result is the following.

\begin{theorem}\label{th:firstvariation}
  Let $\K$ be a convex body in a Minkowski space $(V,\|\cdot\|_{\B})$.
  Then for any $v \in V$ we have
   \begin{equation}\label{eq:firstvariation}
     \left.\frac{d}{dt}\right|_{t=0}\area_{\B^\ast}({\partial(\K-tv)^\ast})= \frac{n+1}{\eps_{n-1}}\left\langle  \CC_{\B}(\K^\ast),v\right\rangle.
  \end{equation}
  In particular $\SS_{\B}(\K)$ lies at the origin of $V$ if and only if $\CC_{\B}(\K^\ast)$ lies at the origin of $V^\ast$.
\end{theorem}

Here $\eps_{n-1}$ denotes the Euclidean volume of an $(n-1)$-dimensional unit Euclidean ball and $\langle\cdot,\cdot\rangle$  the canonical duality pairing. Note the similarity with formula (\ref{eq:var_vol_dual}).

In the particular case of a convex body in an Euclidean space, both theorems sum up into the following.

\begin{corollary}
  Given a convex body $\K$ in $\R^n$, there exists a unique point $\SS(\K) \in \inter{\K}$ minimizing the functional
  \[
  x \in \inter{\K} \mapsto \hausdorff{n-1}(\partial(\K-x)^\circ).
  \]
  Furthermore, this point lies at the origin if and only if the following average of centroids of orthogonal projections of its polar body
  \[
    \CS(\K^\circ)=\int_{\S^{n-1}}\left(\int_{\pi_{x^\perp} (\K^\circ)} \, \, y \,  d\lambda(y)\right) \, d\sigma(x)
  \]
  lies at the origin.
\end{corollary}

Here $\pi_{x^\perp}$ denotes the orthogonal projection onto the hyperplane  orthogonal  to $x$, and $\sigma,\lambda$ denote respectively the canonical measure on $\S^{n-1}$ and the Lebesgue measure on $x^\perp$  induced by the standard Euclidean structure.

This paper is organized as follows. In section \ref{sec:rappels} we collect some material on the Holmes-Thompson volume, and present a Crofton formula for the Holmes-Thompson area of a hypersurface in a Minkowski space due to \cite{Alv98}. This formula will be decisive to show (\ref{eq:firstvariation}). In section \ref{sec:convexity} we prove Theorems \ref{th:santalo} and \ref{th:santaloK=B}. The next three sections are devoted to the proof of Theorem \ref{th:firstvariation}. The strategy consists of using the above mentioned Crofton formula which describes the Holmes-Thompson area of $ \partial \K^*$ as an integral on the set of lines intersecting $\K^\ast$. In order to compute its first variation under translations, we first use the boundary sphere $\partial \B$ to construct a parametrization space for the oriented affine lines of the dual vector space in terms of intersections of affine hyperplanes. We then show how to rewrite the measure involved in \'Alvarez Paiva's formula in this new space. The effect of translations is then simple to describe which  allows us to compute the first variation of the area.  So after shortly recalling in section \ref{sec:equiaffine} some classical notions of equiaffine differential geometry, we define the parametrization space in section \ref{sec:pullback} and show how to rewrite the measure of lines in this new setup. In the last section \ref{sec:firstvariation} we compute the first variation of the Holmes-Thompson area and obtain formula (\ref{eq:firstvariation}).

\noindent {\bf Acknowledgements.} We would like to thank the referee for valuable comments.\\

%%%%%%%%%%%%%%%%%%%%%%%%%%%%%%%%%%%%%
\section{Background on the Holmes-Thompson notions of volume and area}\label{sec:rappels}
%%%%%%%%%%%%%%%%%%%%%%%%%%%%%%%%%%%%%

In this section we collect some material on the Holmes-Thompson notions of volume and area. The two major tools we will use in the sequel are a Crofton type formula due to \'Alvarez Paiva and presented in Proposition \ref{prop:AP}, and a duality formula due to Holmes \& Thompson stated in Proposition \ref{prop:HT}.

%%%%%%%%%%%%%%%%%%%%%%%%%%%%%%%%%%%%%
\subsection{Holmes-Thompson volume on continuous Finsler manifolds}
%%%%%%%%%%%%%%%%%%%%%%%%%%%%%%%%%%%%%

Let $M$ be a smooth manifold $M$ of dimension $k$ endowed with a continuous Finsler metric $F:TM \to [0,\infty)$, that is a continuous function which is positively homogenous of degree $1$ outside the zero section and such that for all $x\in M$ the subsets of $T_x M$
\[
  D_x(M,F) \coloneqq \{v \in T_x M \mid F(x,v)\leq 1\}
\]
are convex bodies containing the origin in their interior (or equivalently, $F$ is a continuous function whose restriction to each tangent space $T_xM$ is a norm, possibly asymmetric). For any measurable set $A\subset M$ denote by $D^\ast(A,F)$ the unit co-disc bundle over $A$ consisting of the disjoint union of the dual bodies $D_x(M,F)^\ast \subset (T_xM)^\ast$ for $x \in A$. The {\it Holmes-Thompson volume} of the measurable set $A$ in $(M,F)$ is then defined as the following integral:
\[
  \vol(A,F)={\frac{1}{k!\varepsilon_{k}}  \,}\int_{D^\ast(A,F)} \omega_M^{k}
\]
where $\omega_M$ denotes the standard symplectic form on the cotangent bundle $T^\ast M$. If $\rho$ denotes any non-vanishing $k$-density on $M$, then
\[
  \vol(A,F)=\int_A \frac{|D_x(M,F)^\ast|^\ast_{\rho_x}}{\varepsilon_k} \, \rho_x
\]
where $|\cdot|^\ast_{\rho_x}$ is the dual Lebesgue measure on $T_x^\ast M$ of the Lebesgue measure $|\cdot|_{\rho_x}$ on $T_xM$ associated to $\rho$. More precisely, $|\cdot|_{\rho_x}$ is the Lebesgue measure on $T_x M$ normalized to give volume $1$ on a parallelotope spanned by a base $v_1,\ldots,v_k$ of $T_x M$ for which $\rho_x(v_1,\ldots,v_k)=1$, and $|\cdot|^\ast_{\rho_x}$ is the Lebesgue measure on $T_x^\ast M$ normalized to give volume $1$ for the parallelotope spanned by the covectors $\xi_1,\ldots,\xi_k \in T_x^\ast M$ forming a base which is dual to $(v_1,\ldots,v_k)$.

%%%%%%%%%%%%%%%%%%%%%%%%%%%%%%%%%%%%%
\subsection{Holmes-Thompson volume and area in normed spaces}
%%%%%%%%%%%%%%%%%%%%%%%%%%%%%%%%%%%%%

Let ${\D}$ be a convex body in an $n$-dimensional real vector space $V$ that contains the origin as an interior point.  The associated norm $\|\cdot\|_{\D}$ defines a continuous Finsler metric on the manifold $V$ whose Holmes-Thompson volume satisfies for any compact set $A$ the equality
\[
  \vol(A,\|\cdot\|_{\D})=\frac{|A| \, |{\D}^\ast|^\ast}{\eps_n}
\]
where $|\cdot|$ and $|\cdot|^\ast$ are respectively the measure and its dual measure associated to some translation-invariant density on $V$. We will denote shortly $\vol(A,\|\cdot\|_{\D})$ by $\vol_{\D}(A)$.

In the case where $V=\R^n$  we find that $\vol_{\D}(A)=|A| \, |{\D}^\circ|/\eps_n$. In particular the Finsler volume $\vol_{\D}(\D)$ of the convex body ${\D}$ itself coincides with its volume product up to the constant $1/\eps_n$.

It turns out that the associated norm $\|\cdot\|_{\D}$ also permits to define for any hypersurface $M\subset V$ (and, more generally, for any smooth submanifold of $V$) a continuous Finsler metric $F_{\D}:TM\to [0,\infty)$ by simply setting $F_{\D} (x,v)=\|v\|_{\D}.$ We will denote shortly $\vol(M,F_{\D})$ by $\area_{\D}(M)$ and speak of the {\it Holmes-Thompson area} of $M$. This notion of area extends to boundaries of convex bodies. For this recall that  the set of singular boundary points (those for which the supporting hyperplane is not uniquely defined) of a convex body ${\C}$ in $V$ is of $(n-1)$-dimensional Hausdorff measure zero, see \cite{Reid21}. Therefore the formula
\begin{equation}\label{eq:area_HT}
  \area_{\D}(M)=\int_{M} \frac{|(T_xM \cap {\D})^\ast|^\ast_\rho}{\varepsilon_{n-1}}\, \rho(x)
\end{equation}
where $\rho$ denotes any translation-invariant $(n-1)$-density on $V$ makes sense when $M=\partial {\C}$ is the boundary of a convex body ${\C}$. Moreover, if ${\C}_n\to {\C}$ is a sequence of smooth convex bodies that converges in the Hausdorff metric to a convex body ${\C}$, then $\area_{\D}(\partial {\C}_n) \to \area_{\D}(\partial {\C})$ (see Remark \ref{rem:cvHT}).

%%%%%%%%%%%%%%%%%%%%%%%%%%%%%%%%%%%%%%%%%%%%%%%%%%%%
\subsection{Crofton formula for hypersurfaces in Minkowski spaces}\label{sec:AlvarezFormula}
%%%%%%%%%%%%%%%%%%%%%%%%%%%%%%%%%%%%%%%%%%%%%%%%%%%%

In case the normed space ${(V,\|\cdot\|_{\D})}$ has smoothness and strict convexity properties, the Holmes-Thompson area of hypersurfaces admits a Crofton formula due to Alvarez-Paiva (cf. \cite{Alv98}) that we present next. We refer to \cite{Sch06} for a very complete survey on the integral geometry of normed spaces.

Let ${\D}$ be a convex body of $V$ that contains the origin as an interior point, and suppose $\partial \D$ is smooth and positively curved. Then the associated normed space $(V,\|\cdot\|_{\D})$ is said to be {\it Minkowski}. In that case its dual body also has smooth and positively curved boundary, and the dual Legendre transform
\[
  {\LL}^\ast : \partial {\D}^\ast \to \partial {\D}
\]
is a well defined diffeomorphism. Recall that ${\LL}^\ast(p)$ is defined as the unique $x \in \partial {\D}$ such that $p(x)=1$. Denote by $G_+(V)$ the space of oriented affine lines in $V$ whose elements are denoted by $x+\langle v \rangle_+$ where $x$ is a point in $V$ and  $\langle v \rangle_+$ is the $1$-dimensional vector subspace of $V$ spanned and oriented by a vector $v$. In particular according to this notation $x+\langle v \rangle_+=y+\langle w \rangle_+$ if and only if there exists $(t,s)\in \R\times \R_{>0}$ such that $y= x+t v$ and $w=sv$. Now consider the projection map
\begin{align*}
  \pi : V\times \partial {\D}^\ast &                \to G_+(V) \\
               (x,p)               & \mapsto x+\langle{\LL}^\ast(p)\rangle_+.
\end{align*}
Recall that $\omega_V$ denotes the standard symplectic form of $T^\ast V$ and denote by $i : V \times \partial {\D}^\ast \hookrightarrow T^\ast V$ the canonical inclusion. It is well known that there exists a unique symplectic form $\omega_{\D}$ on $G_+(V)$ such that $\pi^\ast \omega_{\D}=i^\ast \omega_V$, see \cite{AG90,Besse}. Here is the Crofton formula associated to this symplectic form.

\begin{proposition}[\'Alvarez Paiva]\label{prop:AP}
The Holmes-Thompson area of a compact immersed hypersurface $M$ in a Minkowski space $(V,\|\cdot\|_{\D})$ satisfies the following formula:
\[
  \area_{\D}(M)=\frac{1}{2 \, (n-1)! \, \eps_{n-1}} \int_{G_+(V)} \# (L\cap M)\,  \omega_{\D}^{n-1}.
\]
\end{proposition}

Applying this formula to the boundary of a smooth convex ${\C}$ in $V$ we obtain the following formula:
\[
  \area_{\D}(\partial {\C})=\frac{1}{(n-1)! \, \eps_{n-1}} \int_{\{L\colon L\cap {\C}\neq \emptyset\}}  \omega_{\D}^{n-1}.
\]
By continuity we see that this formula still holds when the convex body ${\C}$ is not necessarily smooth.

Proposition \ref{prop:AP} was first stated in  \cite{Alv98}  (see \cite[Theorem 3.1]{AF98}) more generally for immersed hypersurfaces in reversible Finsler manifolds whose space of oriented geodesics is a smooth manifold. We include the proof here for the reader's convenience as well to check that it is still true in our context without the reversibility assumption.

\begin{proof}
First note that if $H\subset V$ is a linear subspace, then $({\D}\cap H)^\ast=r_{H^\ast}({\D}^\ast)$ where $r_{H^\ast} : V^\ast \to H^\ast$ denotes the restriction morphism defined by $r_{H^\ast}(p)=p_{|H}$. Thus
\begin{align*}
  D^\ast(M,F_{\D})
    & =\{(x,p)\in T^\ast M \mid x \in M \, \text{and} \, \, p \in ({\D}\cap T_xM)^\ast\}                 \\
    & =\{(x,p)\in T^\ast M \mid x \in M \, \text{and} \, \, p \in r_{(T_xM)^\ast} ({\D}^\ast)\}          \\
    & =\{(x,p)\in T^\ast M \mid x \in M \, \text{and} \, \, p \in r_{(T_xM)^\ast} (\partial {\D}^\ast)\} \\
    & =P_M(M\times \partial {\D}^\ast)
\end{align*}
where
\begin{align*}
  P_M : M \times \partial {\D}^\ast &   \to D^\ast(M,F_{\D}) \\
                              (x,p) & \mapsto (x,p_{|T_xM}).
\end{align*}
Therefore, by the coarea formula
\[
  (n-1)! \, \eps_{n-1} \, \area_{\D}(M)= \int_{D^\ast(M,F_{\D})} \omega_M^{n-1}={\frac{1}{2}} \, \int_{M\times \partial {\D}^\ast} |P^\ast_M \omega_M^{n-1}|.
\]
Here $|\eta|$ denotes the density given by the absolute value of a top differential form $\eta$.
Now observe that $P^\ast_M \omega_M=i^\ast\omega_V$ (as both are the exterior derivative of the tautological one-form) which implies, using the identity $i^\ast\omega_V=\pi^\ast\omega_{\D}$ and the coarea formula again, that
\[
  2 \, (n-1)! \, \eps_{n-1} \,\area_{\D}(M)=\int_{M\times \partial {\D}^\ast} |\pi^\ast\omega_{\D}^{n-1}|= \int_{G_+(V)} \# (L\cap M)\omega_{\D}^{n-1}.
\]
\end{proof}

Observe that we obtain by the way the following formula for Holmes-Thompson area:
\begin{equation}\label{eq:symplecticHT}
  \area_{\D}(\partial {\C})=\frac{1}{2\, (n-1)! \, \eps_{n-1}} \int_{\partial {\C}\times \partial {\D}^\ast} |i^\ast\omega_V^{n-1}|.
\end{equation}

\begin{remark}\label{rem:cvHT}
We easily check  from the above formula that if ${\C}_n, {\D_n}$ are sequences of convex bodies that converge in the Hausdorff metric to ${\C,\D}$ respectively, then $\area_{\D_n}(\partial {\C}_n) \to \area_{\D}(\partial {\C})$.
\end{remark}

%%%%%%%%%%%%%%%%%%%
\subsection{Holmes-Thompson duality formula}
%%%%%%%%%%%%%%%%%%%

To conclude this section let us recall the following duality principle for the Holmes-Thompson area of convex bodies.

\begin{proposition}[Holmes \& Thompson]\label{prop:HT}
Let ${\C}$ and ${\D}$ be two convex bodies in $V$ that contain the origin in their interior. Then
\[
  \area_{{\D}^\ast}(\partial {\C}^\ast)=\area_{\C}(\partial {\D}).
\]
\end{proposition}

\begin{proof}
This was first observed by Holmes and Thompson in \cite{HT79}. Here is a quick argument:
\begin{align*}
  \area_{\C}(\partial {\D})
    & = \frac{1}{2\, (n-1)! \, \eps_{n-1}} \int_{\partial {\D}\times \partial {\C}^\ast} |\omega_V^{n-1}|                     \\
    & = \frac{1}{2\, (n-1)! \, \eps_{n-1}} \int_{\partial {\C}^\ast \times \partial ({\D}^\ast)^\ast} |\omega_{V^\ast}^{n-1}| \\
    & =  \area_{{\D}^\ast}(\partial {\C}^\ast),
\end{align*}
where we have abusively denoted $i^\ast \omega_V$ by $\omega_V$.
\end{proof}

%%%%%%%%%%%%%%%%%%%%%%%%%%%%%%%%%%%%%
\section{Strict convexity and properness}\label{sec:convexity}
%%%%%%%%%%%%%%%%%%%%%%%%%%%%%%%%%%%%%

In this section we prove both Theorem \ref{th:santalo} and Theorem \ref{th:santaloK=B} which we merge in a same statement as follows using Proposition \ref{prop:HT}.

\begin{theorem} \label{teo:convex_proper}
  Let $\B$ and $\K$ be two convex bodies of a finite-dimensional real vector space $V$. Suppose that 
 either (i) $\partial \B$ is of class $C^1$,  or (ii) $K=B$.
  The functional
  \[
    x\mapsto \area_{\K-x}(\partial \B)
  \]
  is strictly convex and proper on the interior of $\K$.

  In particular there exists a unique minimizing point $\SS_{\B}(\K) \in \inter{\K}$ for the Holmes-Thompson area of the boundary sphere of $(\K-x)^\ast$ in $(V^\ast,\|\cdot\|_{\B^\ast})$.
\end{theorem}

We split the proof into two subsections. The first concerns the strict convexity of the functional above. The second proves the properness of the functional, using an isoperimetric inequality between Holmes-Thompson notions of volume and area. In a third subsection we explain why the map $\D \mapsto \SS_{\D}(\D)$ gives rise to an affine-invariant point. Finally, in the last subsection we further explore conditions which ensure the strict convexity beyond conditions (i) and (ii).

%%%%%%%%%%%%%%%%%%%
\subsection{Proof of the strict convexity}
%%%%%%%%%%%%%%%%%%%

After fixing some isomorphism $V\simeq \R^n$, we can use the standard Euclidean structure to identify the dual body $A^\ast$ with the polar body $A^\circ$ of any convex body $A$ containing the origin as an interior point, and rewrite
\begin{equation}\label{eq:area_HT_int}
  \area_{\K-x}(\partial \B)=\int_{\partial \B} \frac{|(T_y\partial \B \cap (\K-x))^\circ|_{n-1}}{\varepsilon_{n-1}}d\lambda_{n-1}(y)
\end{equation}
using formula (\ref{eq:area_HT}). Here we have denoted by
\begin{itemize}
  \item  $d\lambda_{n-1}$ the $(n-1)$-dimensional volume element induced on hypersufaces by the Euclidean structure
  \item $|C|_{n-1}=\int_Cd\lambda_{n-1}$ the  $(n-1)$-Hausdorff measure of any compact domain $C$ contained in some hyperplane.
\end{itemize}

Recall that a point $y$ in the boundary of a convex body $\B$ is said to be \emph{regular} if there is a unique supporting hyperplane $y + H(y,\B)$ of $\B$ at $y$.  The set $\regpts{\B}\subset \partial \B$ of regular points has $(n-1)$-Hausdorff measure $\hausdorff{n-1}(\regpts{\B})=\hausdorff{n-1}(\partial \B)$ (see \cite{Reid21}), and the map
\[
  y\in \regpts{\B} \mapsto H(y,\B) \in \Gr{\R^n}
\]
is continuous (cf. \cite[Lemma 2.2.12]{Schn13}). Here $\Gr[k]{\R^n}$ denotes  the set of all vector $k$-planes in $\R^n$. Hence, the push-forward of the $(n-1)$-Hausdorff measure $\hausdorff{n-1}$ by this map is well defined and we denote by $\mu_{\partial \B}$ this measure on $\Gr{\R^n}$. The measure $\mu_{\partial \B}$ coincides with the push-forward of the classical {\em surface area measure} of $\B$ (which is a measure on the sphere $\S^{n-1}$) using the standard two folding map $\S^{n-1} \to \Gr{\R^n}$.

Consider the functional
\begin{align*}
  f_{\K}:\Gr{\R^n} \times\inter{\K} & \to (0,\infty) \\
                                  (H,x) & \mapsto |(H\cap(\K-x))^\circ|_{n-1}
\end{align*}
which is continuous as a composition of continuous maps.

\begin{lemma}\label{lem:convexity}
  For any $H \in \Gr{\R^n}$, the function $x\mapsto f_{\K}(H,x)$ is convex on $\inter{\K}$.
\end{lemma}

\begin{proof}
By \cite[eqs.~(1.52)~\&~(1.53)]{Schn13}, in an $m$-dimensional Euclidean vector space $(W,\langle\cdot, \cdot \rangle)$, the Lebesgue measure of any convex body $C$ can be computed using the  formula
\[
  |C|_m={\frac{1}{m}}\int_{\S^{m-1}(W)} \left( \frac{1}{h_{C^\circ}(u)} \right)^{m}d\lambda_{m-1}(u),
\]
where $\S^{m-1}(W)$ denotes the unit sphere in $(W,\langle\cdot, \cdot \rangle)$ and $h_{C^\circ}$ is the support function of the polar body $C^\circ$, defined for any $x \in W$ by $h_{C^\circ}(x)=\max \{\langle x,y\rangle \mid y \in C^\circ\}$. Therefore
\[
  f_{\K}(H,x)={\frac{1}{n-1}}\int_{\S^{n-2}(H)}  \left(\frac{1}{h_{H\cap(\K-x)}(u)} \right)^{n-1}d\lambda_{n-2}(u).
\]
Now observe that $\forall u \in \S^{n-2}(H)$ the function $x\mapsto h_{H\cap(\K-x)}(u)$ is concave as for any $\lambda \in (0,1)$ and $x_1,x_2 \in \inter{\K}$ we have
\begin{align*}
  h_{H \cap (\K-\lambda x_1-(1-\lambda)x_2)}(u)
    & =\max \{\langle u,z\rangle \mid z \in H\cap [\lambda (\K-x_1)+(1-\lambda)(\K-x_2)]\}              \\
    & \geq \max \{\langle u,z\rangle \mid z \in \lambda[ H\cap (\K-x_1)]+(1-\lambda)[H\cap (\K-x_2))]\} \\
    & =\lambda h_{H\cap (\K-x_1)}(u)+(1-\lambda) h_{H\cap (\K-x_2)}(u).
\end{align*}
The function $t\mapsto {1 /t^{n-1}}$ being strictly convex on $(0,\infty)$, we get that the function
\[
  x\mapsto \left( \frac{1}{h_{H\cap(\K-x)}(u)} \right)^{n-1}
\]
is convex. So $f_{\K}(H,\cdot)$ is also convex as an integral of convex functions.
\end{proof}

For any $x_{1}\neq x_{2} \in \inter{\K}$ let us define $\Strcvx[\K]{x_1,x_2} \subset \Gr{\R^n}$ as the set of hyperplanes $H$ such that
\begin{equation}\label{eq:strict}
   f_{\K}(H,x_1) + f_{\K}(H,x_2) - 2 \, f_{\K}(H,(x_1 + x_2)/2) > 0.
\end{equation}
Observe that $\Strcvx[\K]{x_1,x_2}$ is an open set by continuity of $f_{\K}$. Furthermore, for a fixed $H$, we know by \cite{San49} that for any $x_0\in \inter{\K}$ the map $x\mapsto f_{\K}(H,x+x_0)$ is strictly convex on $H\cap\inter{\K-x_0}$. So in particular, we always have $\Strcvx[\K]{x_1,x_2}\neq \emptyset$ as it contains any hyperplane $H$ containing  $x_2-x_1$. Now let us say that $(x_1,x_2) \in \inter{\K}\times \inter{\K}\setminus \Delta$ is a \emph{cylindrical pair} for $\K$ if $\Strcvx[\K]{x_1,x_2}\neq \Gr{\R^n}$. Here $\Delta$ denotes the diagonal subset. We denote by $\cyl{\K} \subset  \inter{\K}\times \inter{\K}\setminus \Delta$ the subset of cylindrical pairs associated to $\K$. The term cylindrical pair is justified by the following result.

\begin{lemma}\label{lemma:realcyl}
  For any $(x_1,x_{2}) \in \cyl{\K}$, we can find a hyperplane $H\in \Gr{\R^n}$ and a convex body $C$ of $H$ such that
  \[
    H\cap(\K-x)=C
  \]
  for any $x \in [x_1,x_2]$.
   Equivalently, we have that $(H+ [x_1,x_2])\cap \K=C+ [x_1,x_2]$.
\end{lemma}

\begin{proof}
Picking $H \notin \Strcvx[\K]{x_1,x_2}$ we get that
\[
 f_{\K}(H,\lambda x_1 +(1-\lambda)x_2)= \lambda f_{\K}(H,(x_1)+(1-\lambda) f_{\K}(H,x_2)
\]
for $\lambda=1/2$ and thus for all $\lambda\in(0,1)$. It follows by the proof of Lemma \ref{lem:convexity} that  $\forall u \in \S^{n-2}(H)$ and $\forall \lambda \in [0,1]$
\[
  h_{H\cap(\K-x_1)}(u)=h_{H\cap(\K-\lambda x_1 - (1-\lambda)x_2)}(u).
\]
Therefore there exists a convex body $C$ of $H$ such that $H\cap(\K-z)=C$ for any $z \in [x_1,x_2]$. Equivalently, we have that $(H+ [x_1,x_2])\cap \K=C+ [x_1,x_2]$.
\end{proof}

The following result describes how cylindrical pairs of $\K$ should interact with the support of $\mu_{\partial \B}$ to ensure strict convexity for the Holmes-Thompson area. Recall that the support $\supp{\mu}$ of a Borel measure $\mu$ is the intersection of all closed sets of total measure.

\begin{proposition}\label{str-conv-ab}
  The map $x\mapsto\area_{\K-x}(\partial \B)$ is strictly convex on $\inter{\K}$ if and only if $\supp{\mu_{\partial \B}} \cap \Strcvx[\K]{x_1,x_2} \neq \emptyset $ for every pair $(x_1,x_2) \in \cyl{\K}$.
\end{proposition}

Observe  that $\supp{\mu_{\partial \B}} \cap \Strcvx[\K]{x_1,x_2} \neq \emptyset$ for every $(x_1,x_2) \notin \cyl{\K}$ as $\Strcvx[\K]{x_1,x_2}=\Gr{\R^n}$ in this case. So we can replace ``for every pair $(x_1,x_2) \in \cyl{\K}$'' by ``for every $x_1\neq x_2 \in \inter{\K}$'' in the proposition above.
\begin{proof}
The integral of a non-negative continuous function against a Borel measure is positive if and only if the support of the measure meets the interior of the support of the function. Therefore, fixing $ x_1\neq x_2 \in \inter{\K}$, we get that
\begin{multline*}
  \area_{\K-x_1}(\partial \B) +  \area_{\K-x_2}(\partial \B) - 2 \, \area_{\K-(x_1 + x_2)/2}(\partial \B) \\
  = \frac{1}{\eps_{n-1}}\int_{\Gr{\R^n}}\left[ f_{\K}(H,x_1) + f_{\K}(H,x_2) - 2 \, f_{\K}(H,(x_1 + x_2)/2)\right] d\mu_{\partial \B}(H)
\end{multline*}
is positive if and only if $\Strcvx[\K]{x_1,x_2}$ meets $\supp{\mu_{\partial \B}}$.
\end{proof}

We are now ready to prove the strict convexity part of Theorem \ref{teo:convex_proper}.

\begin{proof}[\textup{\textbf{Proof of the strict convexity in Theorem \ref{teo:convex_proper}}}]

(i) Suppose first that $\partial \B$ is of class $C^1$. Thus $\supp{\mu_{\partial \B}} = \Gr{\R^n}$ which implies that the map $x\mapsto\area_{\K-x}(\partial \B)$ is always strictly convex on $\inter{\K}$ according to Proposition \ref{str-conv-ab}.

(ii) Suppose now that $\K=\B$.
By Proposition \ref{str-conv-ab}  we need to prove that $\supp{\mu_{\partial \B}}\cap \mathrm{Strcvx}_{\B}(x_1,x_2)\neq \emptyset$ for every $(x_1,x_2)\in \cyl{\B}$.
 
 By Lemma \ref{lemma:realcyl}, given any $(x_1,x_2)\in \cyl{\B}$, there exists $H\in\Gr{\R^n}$ and $C$ a convex body in $H$ such that $(H+[x_1,x_2])\cap \B=C+[x_1,x_2]$. Let $\mu_{\partial C}$ be the push-forward measure of the surface area measure of $C$ by the quotient map $S^{n-2}(H)\to \Gr[n-2]{H}$. Take any $G\in\supp{\mu_{\partial C}}$, and consider the hyperplane $H'=G\oplus \mathrm{span}(x_2-x_1)$. Since $x_2-x_1\in H'$, we know by [Santal\'o,1949] that $H'\in\mathrm{Strcvx}_{\B}(x_1,x_2)$. Let  $N$ be any open neighborhood of $H'$ in $\Gr{\R^n}$. Let us prove that $\mu_{\partial \B}(N)>0$.  
 Consider the map
 \begin{eqnarray*}
  \iota:\Gr[n-2]{H}&\to &\Gr{\R^n}\\
  F&\mapsto & F\oplus \mathrm{span}(x_2-x_1)
\end{eqnarray*}
which is continuous.
Since $\iota^{-1}(N)$ is open and contains $G\in\supp{\mu_{\partial C}}$ we have $\mu_{\partial C}(\iota^{-1}(N))>0$. Let $U$ be the set of regular points of $\partial \B$ with supporting hyperplane inside $N$, and let $V\subset \partial C$ consist of the regular points of $\partial C$ with supporting hyperplane in $\iota^{-1}(N)$. By  \cite[Theorems 25.1 and 25.5]{Rock97}, there exists an injective $C^1$-map $\varphi$ defined on some measurable set $D\subset \R^{n-2}$ and such that $\varphi(D)=V$. 
Set
\begin{eqnarray*}
\bar{\varphi} : D \times ]x_1,x_2[&\to & V+]x_1,x_2[\subset U\\
(y,x)&\mapsto&\varphi(y)+x.
\end{eqnarray*}
Observe that the Jacobians of $\varphi$ and $\bar{\varphi}$ satisfy the inequality $\mathrm{Jac}_x \, \bar{\varphi}\geq \mathrm{Jac}_x \, \varphi \cdot \sin \alpha$
where $\alpha$ denotes the angle between $x_2-x_1$ and $H$.
Therefore we have
\begin{align*}
\mu_{\partial \B} (N) =\hausdorff{n-1}(U)
&\geq \int_{D\times]x_1,x_2[}  \mathrm{Jac}_x\bar{\varphi}\\
&\geq |x_2-x_1| \cdot \sin\alpha \cdot \int_D  \mathrm{Jac}_x\varphi\\
 &= | x_2-x_1|\cdot\sin(\alpha)\cdot \hausdorff{n-2}(V)\\
 &= | x_2-x_1|\cdot\sin(\alpha)\cdot \mu_{\partial C}(\iota^{-1}(N)) >0,
 \end{align*}
as claimed. This implies that $H'\in \supp{ \mu_{\partial \B}}$. Since also $H'\in\mathrm{Strcvx}_{\B}(x_1,x_2)$, this proves strict convexity in case (ii).
\end{proof}

\begin{remark}\label{rem:cex}
  We now present an example of  two convex bodies $\K$ and $\B$ in the plane $\R^2$ such that the functional $x\mapsto \area_{\K-x}(\partial \B)$ is convex but not strictly convex:
  \begin{center}
      \includegraphics{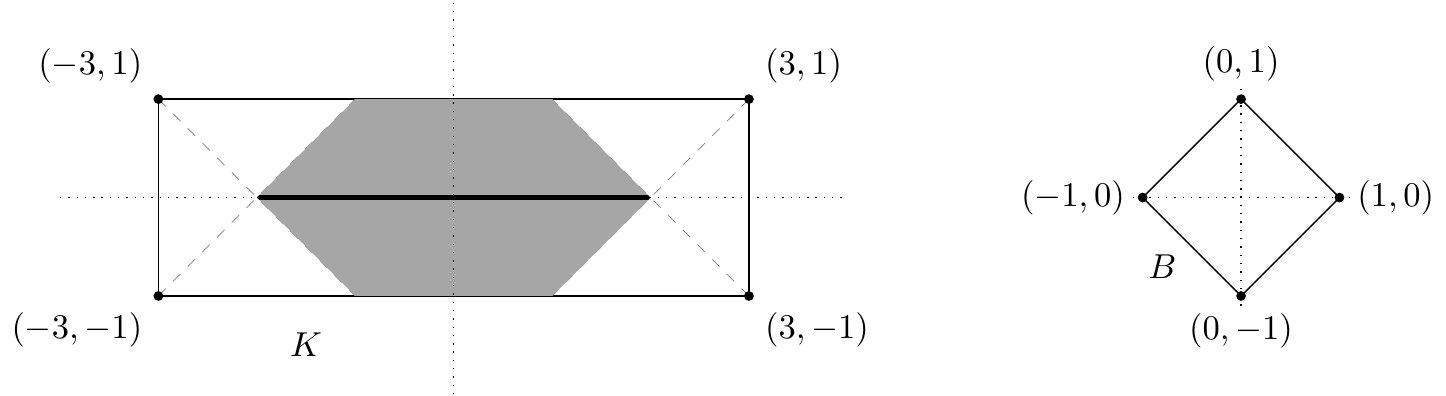}
  \end{center}
  Observe that $\supp{\mu_{\partial \B}}$ consists of the two directions ${x=y},{x=-y}$ each of weight $2\sqrt{2}$. It is easy to check that the interior of the shaded area corresponds to points $(x,y)$ where the functional $t\mapsto \area_{\K-(x+t,y)}(\partial \B)$ is constant for small values of $t$. By symmetry we deduce that this functional achieves its minimum on any point lying on the black segment.
\end{remark}

%%%%%%%%%%%%%%%%%%%
\subsection{Proof of the properness}
%%%%%%%%%%%%%%%%%%%

In order to prove Theorem \ref{teo:convex_proper}, it remains to show that the map is proper on the interior of $\K$. This holds without any assumption on the convex body $\B$, and easily follows from the well-known fact that the map $x\mapsto \vol_{\K-x}(\B)$ is proper on the interior of $\K$ and the following generalization of the classical isoperimetric inequality to asymmetric Minkowski spaces.

\begin{proposition}[Isoperimetric inequality for Holmes-Thompson volume]\label{prop:Isop}
  Let $\B$ and $\K$ be two convex bodies of an $n$-dimensional real vector space $V$. Suppose that $\K$ contains the origin as an interior point. Then
  \[
    \frac{\area_{\K}(\partial \B)^n}{\vol_{\K}(\B)^{n-1}}\geq \frac{(4n)^n}{n!\eps_n}.
  \]
\end{proposition}

This inequality is stated in \cite[Theorem 6.6.4]{Tho96} in the case where $\K$ is symmetric. The proof straightforwardly generalizes to the non-symmetric case so we briefly survey the arguments.

\begin{proof}
After fixing some isomorphism $V\simeq \R^n$, we identify the dual body $A^\ast$ of any convex body $A$ containing the origin as an interior point with its polar body $A^\circ$. First recall that (compare with \cite[Theorem 5.2.2]{Tho96})
\[
  \area_{\K}(\partial \B)=n \cdot V(\B[n-1],I_{\K})
\]
where $V$ denotes the mixed volume and $I_{\K}$ the isoperimetrix convex body defined as the unique symmetric convex body with support function
\begin{equation}\label{eq:isoperimetrix}
  h_{I_{\K}}(u)=|\pi_{u^\perp}(\K^\circ)|/\eps_{n-1}=|(\K\cap u^\perp)^\circ|/\eps_{n-1}.
\end{equation}
Observe that  $I_{\K}$ is a zonoid according to \cite[Formula 5.80 and Theorem 3.5.3]{Schn13}.

Next we check that
\[
  \frac{\area_{\K}(\partial \B)^n}{\vol_{\K}(\B)^{n-1}}\geq \frac{\area_{\K}(\partial I_{\K})^n}{\vol_{\K}(I_{\K})^{n-1}}.
\]
Indeed as $\area_{\K}(\partial I_{\K})=n|I_{\K}|$ and $\vol_{\K}(\B)=|\B||\K^\circ|/\eps_n$, the above inequality is equivalent to
\[
  V(\B[n-1],I_{\K})^n\geq |I_{\K}||\B|^{n-1}
\]
which directly follows from Minkowski's inequality, see \cite[Theorem 7.2.1]{Schn13}.

Now recall that the normalized isoperimetrix convex body $\tilde{I}_{\K}$ is defined as the unique dilated of the isoperimetrix satisfying $n\vol_{\K}(\tilde{I}_{\K})=\area_{\K}(\partial \tilde{I}_{\K})$. Therefore
\[
  \frac{\area_{\K}(\partial \B)^n}{\vol_{\K}(\B)^{n-1}}\geq \frac{\area_{\K}(\partial \tilde{I}_{\K})^n}{\vol_{\K}(\tilde{I}_{\K})^{n-1}}=n^n \vol_{\K}(\tilde{I}_{\K})=n^n \frac{|\K^\circ||\tilde{I}_{\K}|}{\eps_n}.
\]
Besides it is easy to see that in fact  $\tilde{I}_{\K}=\eps_n\cdot I_{\K}/|\K^\circ|$.
Using equation (\ref{eq:isoperimetrix}), we also see that $\eps_{n-1}\cdot h_{I_{\K}}$ is precisely the support function of the projection body $\Pi \K^\circ$. Consequently we find the identity $\tilde{I}_{\K}=(\eps_{n}/\eps_{n-1}|\K^\circ|)\cdot \Pi \K^\circ$. Therefore
\[
  |(\tilde{I}_{\K})^\circ|=\left(\frac{\eps_{n-1}|\K^\circ|}{\eps_n}\right)^n|(\Pi \K^\circ)^\circ|=\left(\frac{\eps_{n-1}}{\eps_n}\right)^n|\K^\circ||\K^\circ|^{n-1}|(\Pi \K^\circ)^\circ|,
\]
which implies together with Petty's projection inequality \cite{Petty71}, see \cite[Formula 10.86]{Schn13}, that
\[
  |(\tilde{I}_{\K})^\circ| \leq|\K^\circ|.
\]
Finally
\[
  \frac{\area_{\K}(\partial \B)^n}{\vol_{\K}(\B)^{n-1}}\geq n^n \frac{|(\tilde{I}_{\K})^\circ||\tilde{I}_{\K}|}{\eps_n}\geq  \frac{(4n)^n}{n!\eps_n}
\]
using Reisner optimal lower bound on the volume product for zonoids \cite{Reis85}.
\end{proof}

The properness statement in Theorem \ref{teo:convex_proper} is now a consequence of the following simple fact.

\begin{proposition}\label{prop:paral}
 Let $B_k\to B$ be a convergent sequence of convex bodies in $\R^n$, and assume that $0\in\inter{B_k}$ for each $k$ and $0\in\partial B$. Then $|B_k^\circ|\to\infty$.
\end{proposition}
\begin{proof}Given a convex body $K$, let us say that a parallelotope circumscribes $K$ if $K\subset P$ and every face of $P$ is contained in a supporting hyperplane of $K$.
Without loss of generality assume that $B$ is supported by a coordinate hyperplane  at $0$. Let $P_k$ (resp. $P$) be the parallelotope circumscribing $B_k$ (resp. $B$), and with all faces parallel to the coordinate hyperplanes. Then  $P_n^\circ$ is a sequence of polytopes with vertices contained in the coordinate lines. One of these vertices tends to infinity while the others converge to points different from the origin. It follows that
 \[
  |B_n^\circ|\geq |P_n^\circ|\to\infty.
 \]
\end{proof}

%%%%%%%%%%%%%%%%%%%%%%%%%
\subsection{A new affine-invariant point}\label{sec:affineinvariantpoint}
%%%%%%%%%%%%%%%%%%%%%%%%%

Let us start by recalling the definition of an {\it affine-invariant point} (see \cite{Grun63}): this is a map $f : {\pazocal K}(\R^n)\to \R^n$ where ${\pazocal K}(\R^n)$ denotes the space of convex bodies in $\R^n$  satisfying the following two conditions
\begin{itemize}
\item for every inversible affine map $\Phi : \R^n \to \R^n$ and every convex body $\B \in  {\pazocal K}(\R^n)$, one has $f(\Phi(\B)) = \Phi(f(\B))$.
\item $f$ is continuous with respect to the Hausdorff metric.
\end{itemize}

We will prove the following.

\begin{proposition}\label{prop_affine_invariant}
The map $f_{\SS} : {\pazocal K}(\R^n)\to \R^n$ defined by $f_{\SS}(\B)=\SS_{\B}(\B)$ is an affine-invariant point.
\end{proposition}

\begin{proof}
Using  \eqref{eq:symplecticHT}  we easily check that given any pair $\K,\B \in{\pazocal K}(\R^n)$ we have $\area_{L\K}(\partial L\B)=\area_{\K}(\partial \B)$  for every invertible linear map $L\colon \R^n\to \R^n$. So if $T\colon \R^n\to \R^n$ is  an invertible affine map, we get that
  \[
    \area_{\B-x}(\partial \B) =\area_{T\B-T(x)}(\partial \,T\B)
  \]
  for all $x \in\inter{\B}$. Our Santal\'o point thus satisfies the equality
  \[
    \SS_{T\B}(T\B)=T(\SS_{\B}(\B))
  \]
from which we deduce that the map $f_{\SS}$ fulfills the first condition.

In order to check the second condition, we argue as follows. First define a function 
$$
F_{\SS} : {\pazocal K}(\R^n)\times  \R^n \to ]0,+\infty]
$$ 
by setting $F_{\SS}(\B,x)=\area_{\B-x}(\partial \B)$ if $x \in \inter{\B}$ and $F_{\SS}(\B,x)=+\infty$ otherwise. Observe that this function is continuous. Indeed, since $F_{\SS}(\B,x)=F_{\SS}(\B-x,0)$, we only need to check continuity with respect to $B$. As sequentially continuous maps defined on metric spaces are continuous, this follows from Remark \ref{rem:cvHT} and Propositions \ref{prop:Isop} and \ref{prop:paral}. 

Given a convex body $\B$, fix $\K_0 \in {\pazocal K}(\R^n)$ such that $B \subset \inter{\K_0}$. By the Blaschke selection theorem, the subset $V:=\{\K\in {\pazocal K}(\R^n) \mid \K \subset \K_0\} $ is a compact neighbourhood of $\B$ in the Hausdorff topology (cf. \cite[Theorems 1.8.4 and 1.8.6]{Schn13}). Let now $\{\B_n\}_n$ be a sequence of convex bodies  in $V$ converging to $\B$, and consider  the sequence $\{f_{\SS}(B_n)\}_n$, which is contained in $\K_0$ as each $B_n\in V$. Let us show that $\{f_{\SS}(B_n)\}_n$ converges to $f_{\SS}(B)$. By contradiction, suppose that there exists an open neighborhood $N$ of $f_{\SS}(\B)$ and a subsequence of $\{f_{\SS}(B_n)\}_n$ contained in $\K_0\setminus N$. This set being compact, there exists a subsequence $\{f_{\SS}(B_{n_k})\}_k$ converging to some point $x\neq f_{\SS}(\B)$. Because $F_{\SS}(\B_n,\SS_{\B_n}(\B_n))\leq F_{\SS}(\B_n,y)$ for all $y \in\R^n$, we deduce by continuity that $F_{\SS}(\B,x)\leq F_{\SS}(\B,y)$ for all $y \in\R^n$. Then both $x$ and $f_{\SS}(B)$ would be minimizers of the function $F_{\SS}(B,\cdot)$, thus contradicting Theorem \ref{teo:convex_proper}. Therefore the sequence $\{f_{\SS}(B_n)\}_n$ necessarily converges to $f_{\SS}(\B)$. This proves the continuity of the map $f_{\SS}$ and concludes the proof.
\end{proof}

%%%%%%%%%%%%%%%%%%%%%%%%%
\subsection{Further study of the strict convexity}
%%%%%%%%%%%%%%%%%%%%%%%%%
 
Next, we describe an optimal property on $\B$ that ensures strict convexity for any convex body $\K$.
For this, let us define for any $L \subset \Gr[1]{\R^n}$ the subset
\[
  \Star{L} \coloneqq \{H \in \Gr{\R^n}\mid L\subset H\}.
\]

\begin{proposition}\label{prop_all_K}
  Let $\B$ be a convex body. The map $x\in \inter{\K}\mapsto\area_{\K-x}(\partial \B)$ is strictly convex for any convex body $\K$ if and only if $\supp{\mu_{\partial \B}} \cap \Star{L} \neq \emptyset $ for any $L \in \Gr[1]{\R^n}$.
\end{proposition}

\begin{proof}
Recall that as already observed, for a fixed $H$ and $\K$, we know by \cite{San49} that for any $x_0\in \inter{\K}$ the map $x\mapsto f_{\K}(H,x)$ is strictly convex on $(x_0+H)\cap\inter{\K}$. In particular, for any $x_1\neq x_2 \in \inter{\K}$, we have that $\Star{\vect{x_2-x_1}} \subset \Strcvx[\K]{x_1,x_2}$. So if $\supp{\mu_{\partial \B}} \cap \Star{L} \neq \emptyset $ for any $L \in \Gr[1]{\R^n}$, we directly get that $\supp{\mu_{\partial \B}} \cap\Strcvx[\K]{x_1,x_2} \neq \emptyset$ for any convex body $\K$ and $x_1\neq x_2 \in \inter{\K}$. By Proposition \ref{str-conv-ab}, it ensures that the map $x\in \inter{\K}\mapsto\area_{\K-x}(\partial \B)$ is always strictly convex.

Now suppose that $\supp{\mu_{\partial \B}} \cap \Star{L} = \emptyset$ for some $L \in \Gr[1]{\R^n}$.

\begin{lemma}\label{open-neighborhood}
  The family of open sets $\{\Strcvx[\K]{x_1,x_2}\}_{(\K,x_1,x_2)}$ where $\K$ runs over convex bodies and $x_1,x_2$ are any pair of distinct points of $\inter{\K}$ such that $\vect{x_2-x_1}=L$ is an open neighborhood base of $\Star{L}$.
\end{lemma}

\begin{proof}[Proof of the Lemma]
Fix an open neighborhood $U$ of $\Star{L}$  in $\Gr{\R^n}$. We construct a convex body $\K$ such that for some $x_1\neq x_2 \in\inter{\K}$ with $\vect{x_2-x_1}=L$ we have $\Star{L}\subset  \Strcvx[\K]{x_1,x_2}\subset U$.
\begin{center}
  \includegraphics[width=10cm]{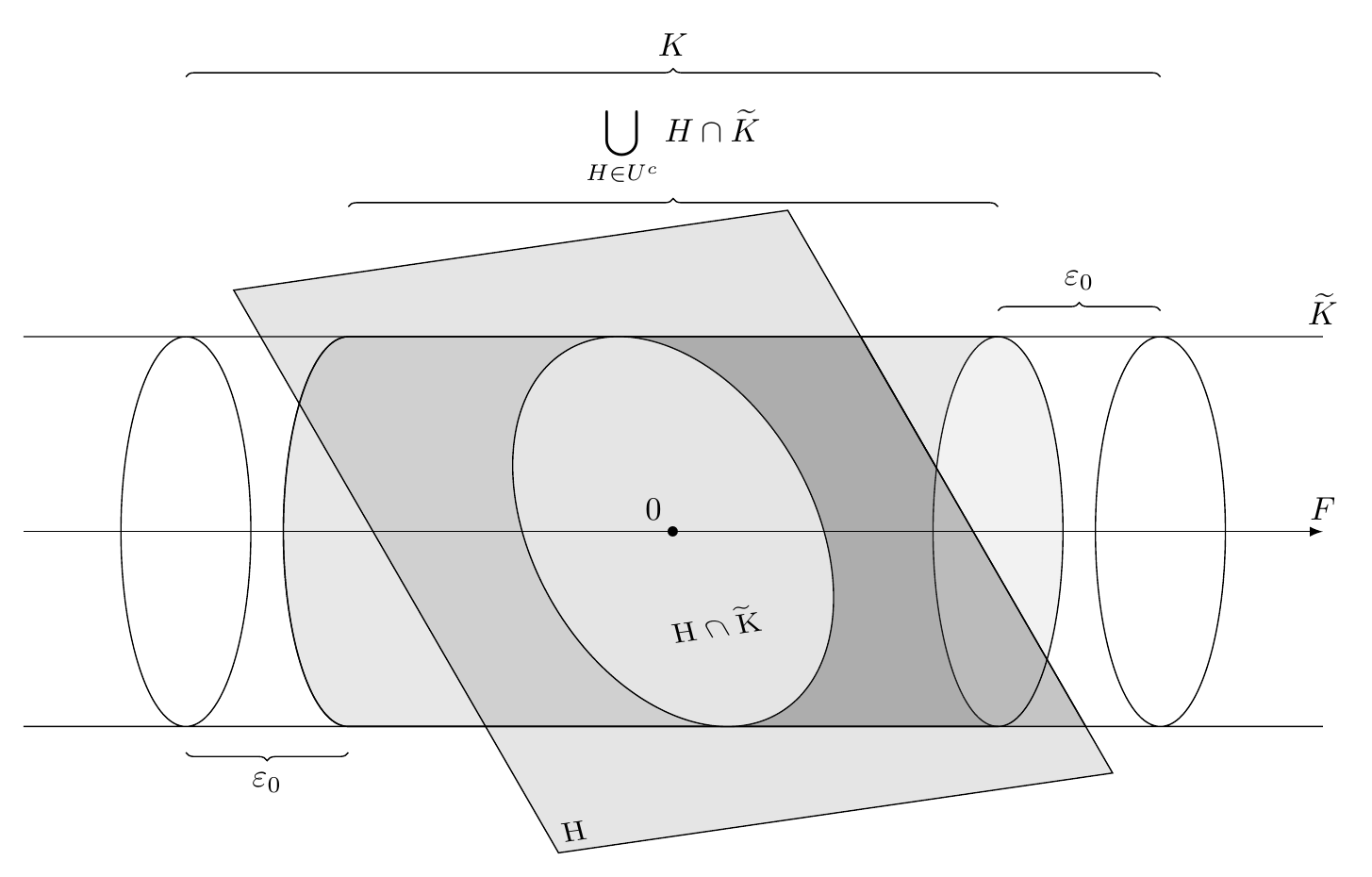}
\end{center}
For this consider the (possibly empty) compact set $\comp{U}$ and let $\widetilde\K$ be any infinite cylinder  with direction $L$, with  a compact convex base and that contains $0$ in its interior. Now fix $\varepsilon_{0} > 0$ and a unit vector $v \in L$. As $\comp{U}$ is compact and disjoint from $\Star{L}$, we can consider a finite subcylinder $\K$ of $\widetilde\K$ delimited by two affine hyperplanes whose underlying directions are not in $\Star{L}$ and such that $(H+\varepsilon v)\cap \widetilde\K = (H+\varepsilon v)\cap \K$ for any $H \in \comp{U}$ and any $|\varepsilon| \leq \varepsilon_{0}$. By taking $x_{1}=-\varepsilon_{0}v$ and $x_{2}=\varepsilon_{0}v$ we have that for any $H \in \comp{U}$, the map $x \mapsto H\cap(\K-x)$ is constant on $[x_1,x_2]$ and so $H \notin \Strcvx[\K]{x_1,x_2}$.
\end{proof}

Lemma \ref{open-neighborhood} applied to the open neighborhood $U \coloneqq \comp{\supp{\mu_{\partial \B}}}$ implies that there exist $\K$ and $x_{1}\neq x_{2}\in \inter{\K}$ such that $\Strcvx[\K]{x_1,x_2}\subset \comp{\supp{\mu_{\partial \K}}}$ and by Proposition \ref{str-conv-ab} we are done.
\end{proof}

To conclude this section, we describe an optimal property on $\K$ that ensures strict convexity for any convex body $\B$.

\begin{proposition}\label{prop_all_B}
  Then the map $x\mapsto\area_{\K-x}(\partial \B)$ is strictly convex on $\inter{\K}$ for any convex body $\B$ if and only if $\K$ has no cylindrical directions, that is $\cyl{\K} = \emptyset$.
\end{proposition}

\begin{proof}
The fact that the condition $\cyl{\K} = \emptyset$ is  sufficient follows directly from \ref{str-conv-ab}. To prove that it is also necessary, let $K$ be a convex body with $\Strcvx[\K]{x_1,x_2}\neq \Gr{\R^n}$ for some cylindrical pair $(x_1,x_2)$. We will construct $B$ such that $\supp{\mu_{\partial \B}} \cap \Strcvx[\K]{x_1,x_2} = \emptyset$. For this we need the following lemma.
\begin{lemma}
  If $\comp{\Strcvx[\K]{x_1,x_2}}$ is nonempty, then it has a nonempty interior.
\end{lemma}
\begin{proof}
  According to Lemma \ref{lemma:realcyl}, if $\Strcvx[\K]{x_1,x_2}\neq \Gr{\R^n}$, we can find a hyperplane $H_0\in \Gr{\R^n}$ and a convex body $C_0$ of $H_0$ such that $H_0\cap(\K-x)=C_0$ for any $x \in [x_1,x_2]$. Fix any proper subinterval $[x'_1,x'_2]$  of $[x_1,x_2]$, that is $[x'_1,x'_2]\subset ]x_1,x_2[$. We can find an open neighbourhood $U$ of $H_{0}$ such that for any $H \in U$ there is a convex body $C$ of $H$ such that $H\cap(\K-x)=C$ for any $x \in [x'_1,x'_2]$. In particular $H\notin \Strcvx[\K]{x'_1,x'_2}\subset \Strcvx[\K]{x_1,x_2}$ which implies that $\comp{\Strcvx[\K]{x_1,x_2}}$ contains an open neighborhood of $H_{0}$.
\end{proof}
From this lemma, it follows that if $\Strcvx[\K]{x_1,x_2}\neq \Gr{\R^n}$, then $\comp{\Strcvx[\K]{x_1,x_2}}$ contains $n$ hyperplanes defined by linearly independent covectors. So, for any parallelotope $\B$ whose face directions are precisely these $n$ hyperplanes, we have that $\supp{\mu_{\partial{\B}}}\cap\Strcvx[\K]{x_1,x_2}=\emptyset$, and so we can conclude by Proposition \ref{str-conv-ab}.
\end{proof}

%%%%%%%%%%%%%%%%%%%%%%%%%%%%%%%%%%%%%
\section{Equiaffine differential geometry}\label{sec:equiaffine}
%%%%%%%%%%%%%%%%%%%%%%%%%%%%%%%%%%%%%

We now introduce classical notions from equiaffine geometry and some notations that will be used to prove Theorem \ref{th:firstvariation} in the next two sections.

From here on, let us suppose the $n$-dimensional vector space $V$ endowed with a non-trivial alternate multilinear $n$-form, which we denote by $\det$.
Let  $M\subset V$ be a hypersurface and take a vector field $\Xi$ transverse to $M$. Then,  for every pair of tangent vector  fields $X,Y$  on $M$, we can decompose the flat connection $\nabla$  of $V$ as
\begin{equation}\label{eq:decomp}
   \nabla_X Y =\nabla^M_X Y +g(X,Y)\cdot\Xi
\end{equation}
where $\nabla^M$ is an affine  torsion-free connection on $M$ and $g$ is a field of symmetric bilinear forms. The hypersurface $M$ is said to be non-degenerate if $g$ is nowhere degenerate, a condition that does not depend on $\Xi$.

\begin{theorem}[\cite{nomizu_sasaki}, Ch.II, Thm.3.1]\label{thm:equiaffine}
  For each non-degenerate oriented hypersurface $M$, there is a unique transversal vector field $\,\Xi$, called  {\em equiaffine normal field} or {\em Blaschke's normal field},  such that 
  \begin{enumerate}
   \item[i)] $\nabla_v \Xi\in T_xM$ for every  $x\in M$ and $v\in T_xM,$
   \item[ii)] the volume $(n-1)$-form $\alpha$ associated to $g$ satisfies
   \[
     \alpha(\xi_1,\ldots,\xi_{n-1})=\det(\Xi(x),\xi_1,\ldots,\xi_{n-1})
   \]
  for every $x\in M$ and $\xi_1,\ldots,\xi_{n-1}\in T_xM$.
  \end{enumerate}
\end{theorem}

 The pseudo-Riemannian metric $g$ and its volume form $\alpha$ are called {\em equiaffine metric}  and {\em equiaffine area measure} respectively.   We assume from here on that $g$ is Riemannian.

We will denote by $E=E(M)$ the frame bundle of $M$ whose fiber over $x\in M$ is
\[
  E_x=\{(x, \xi_1,\ldots, \xi_{n-1})\in M\times V^{n-1}\colon  \vect{\xi_1,\ldots ,\xi_{n-1}}=T_xM \}.
\]
Given $(x, \xi_1,\ldots,\xi_{n-1})\in E$, items i) and ii) in Theorem \ref{thm:equiaffine} read
\begin{equation}\label{eq:conditions}
  \det(\nabla\Xi,\xi_1,\ldots,\xi_{n-1})=0,\quad \det(\Xi(x),\xi_1,\ldots,\xi_{n-1})^{2}=\det (g(\xi_i,\xi_j))_{i,j}.
\end{equation}

\begin{definition}\label{def:L}
  Given an oriented hypersurface $M\subset V$ and $\xi\in E$, let for any $i,j=1,\ldots,n-1$
  \[
    L_{i,j}(\xi) \coloneqq \det(\nabla_{\xi_{i}}{X_j},\xi_1,\ldots,\xi_{n-1}),
  \]
  where $X_j$ is any tangent vector field with ${X_j}(x)=\xi_j$. It is easy to check that $L_{i,j}=L_{j,i}$ and does not depend on $X_j$.
  Let $L\in C^\infty(E)$ be the function defined by
  \begin{equation*}
    L(\xi) \coloneqq \det(L_{i,j}(\xi))_{i,j}.
  \end{equation*}
\end{definition}

\begin{proposition}\label{prop:L_to_det}
  For any  $\xi=(x,\xi_1,\ldots, \xi_{n-1})\in E$ we have
  \begin{equation}\label{eq:L_to_det}
    L(\xi)=\det(\Xi(x),\xi_1,\ldots, \xi_{n-1})^{n+1}.
  \end{equation}
\end{proposition}

\begin{proof}
We can assume $\{\xi_1,\ldots,\xi_{n-1}\}$ to be positively oriented, as both sides of \eqref{eq:L_to_det} are equally affected by a permutation. Indeed, the effect on the matrix $(L_{i,j})_{i,j}$ of a transposition of $\xi_1,\cdots,\xi_{n-1}$ is a sign change of all entries $L_{i,j}$, and a simultaneous permutation of two lines and two columns.
From \eqref{eq:decomp} and \eqref{eq:conditions} we get
\begin{align}
  L_{k,l}(\xi) & =\det(\nabla_{\xi_k} X_l,\xi_1,\ldots,\xi_{n-1})\notag                      \\
               & =g(\xi_k,\xi_l) \det(\Xi(x), \xi_1,\ldots, \xi_{n-1})  \label{eq:L_to_det0} \\
               & =g(\xi_k,\xi_l) \det(g(\xi_i,\xi_j))_{i,j}^{\frac12}.\label{eq:L_to_det1}
\end{align}
It follows using \eqref{eq:conditions} again  that
\begin{equation}\label{eq:L_to_det2}
  L(\xi)=\det(L_{k,l}(\xi))=\det(g(\xi_k,\xi_l))_{k,l}^{1+\frac{n-1}2}=\det(\Xi(x),\xi_1,\ldots, \xi_{n-1})^{n+1}.
\end{equation}
\end{proof}

\begin{remark}
  By \eqref{eq:L_to_det1} and \eqref{eq:L_to_det2}, the equiaffine metric is simply given by
  \[
     g(\xi_i,\xi_j)=|L(\xi)|^{-\frac1{n+1}} L_{i,j}(\xi).
  \]
  In turn, the equiaffine normal vector can be obtained from $g$ as $\Xi=\frac{1}{n-1}\Delta f$ where $\Delta$ is the Laplacian with respect to $g$ and $f\colon M\to V$ is the inclusion (cf. \cite[Thm.6.5, Ch. II]{nomizu_sasaki}). We will not make use of this fact.
\end{remark}

%%%%%%%%%%%%%%%%%%%%%%%%%%%%%%%%%%%%%
\section{Measure of lines in terms of hyperplanes}\label{sec:pullback}
%%%%%%%%%%%%%%%%%%%%%%%%%%%%%%%%%%%%%

In this section we use the boundary sphere $\partial \B$ to construct a parametrization space for the oriented affine lines of the dual vector space in terms of intersections of affine hyperplanes. We then show how to rewrite the measure involved in \'Alvarez Paiva formula in this new space.

Recall that we have fixed some non-trivial alternate multilinear $n$-form det on the $n$-dimensional vector space $V$.

Let $\B$ be a convex body of $V$ that contains the origin as an interior point, and suppose $\partial B$ is smooth and positively curved. In particular its dual body is smooth, and the Legendre transform ${\LL} : \partial \B \to \partial \B^*$, uniquely defined by $$\ker\LL(x)=T_x\partial \B\qquad\mbox{ and }\qquad\langle \LL(x),x\rangle=1,$$ is a well defined diffeomorphism. Moreover, since $\partial \B$ has non-degenerate second fundamental form with respect to any Euclidean structure,  Theorem \ref{thm:equiaffine} applies to it. We orient $\partial \B$ with the  inward vector (i.e. as the boundary of $V\setminus \B$) so that the equiaffine metric $g$ is Riemannian. Associated to the hypersurface $\partial \B$ we also have well defined functions $L_{i,j}$ and $L \in C^\infty(E(\partial \B))$ (see Definition \ref{def:L}).

Consider a smooth local section
\begin{align*}
  \zeta\colon U\subset \partial \B & \to E(\partial \B) \\
                                   x & \mapsto\zeta(x)=(x,\zeta_1(x),\ldots,\zeta_{n-1}(x)).
\end{align*}

We define the following diagram

\begin{center}
  \begin{tikzcd}
 (\R\setminus\{0\})^{n-1}\times U\arrow{r}{F} & E(\partial \B) \arrow{dr}{G}\arrow{d}{H} &          \\
                                              & V^*\times \partial \B\arrow{r}{\pi}      & G_+(V^*)
  \end{tikzcd}
\end{center}
where
\begin{itemize}
  \item $G_+(V^*)$ is the space of oriented affine lines in $V^*$ introduced in section \ref{sec:AlvarezFormula};
  \item for $\lambda=(\lambda_1,\ldots,\lambda_{n-1})\in  (\R\setminus\{0\})^{n-1}$ and $x \in U$
    \[
      F(\lambda,x)=(x,\lambda_1 \zeta_1(x),\ldots,\lambda_{n-1} \zeta_{n-1}(x));
    \]
  \item $G(x,\xi_1,\ldots,\xi_{n-1})$ is the affine line
    \[
      \{p\in V^*\colon \langle\xi_1, p\rangle=\cdots=\langle\xi_{n-1},p\rangle=1\}
    \]
    oriented by $\LL(x)$;
  \item for $(x,\xi_1,\cdots,\xi_{n-1}) \in E(\partial \B)$ let $ \{\LL(x),\xi^1,\ldots,\xi^{n-1}\}$ be the dual basis in $V^*$ of the basis $\{x,\xi_1,\ldots, \xi_{n-1} \}$, and set
    \[
      p(x,\xi_1,\ldots,\xi_{n-1})=\sum_{i=1}^{n-1} \xi^i.
    \]
    Then define
    \[
      H(x,\xi_1,\ldots,\xi_{n-1})=(p(x,\xi_1,\ldots,\xi_{n-1}),x);
    \]
  \item $\pi$ is the projection map
    \[
      \pi(p,x)=p+\langle{\LL}(x)\rangle_+.
    \]
\end{itemize}

It is easy to check that the previous diagram commutes.

Indeed observe that $\langle \xi_i,p(x,\xi_1,\ldots,\xi_{n-1})\rangle=1$ and $\langle\xi_i,\LL (x)\rangle=0$ for all $i$, which leads to the commutativity property:
\[
  \pi \circ H(x,\xi_1,\ldots,\xi_{n-1})=p(x,\xi_1,\ldots,\xi_{n-1})+\langle{\LL}(x)\rangle_+=G(x, \xi_1 ,\ldots, \xi_{n-1}).
\]

Recall that $G_+(V^\ast)$ is endowed with a symplectic form $\omega_{\B^\ast}$ (see section \ref{sec:AlvarezFormula}). It will be convenient to consider on $G_+(V^*)$ the following associated volume element
\[
 \eta_{\B^\ast}=(-1)^{\frac{n(n-1)}2} \omega_{\B^\ast}^{n-1}
\]
and take the corresponding orientation. Our goal in this section is to compute  the pull-back form $G^\ast \eta_{\B^\ast}$. For this, we introduce the following $2$-forms.

\begin{definition}
  Let $\omega_1,\ldots,\omega_{n-1}$ be the $2$-forms on $E(\partial \B)$ given by
  \[
    {\omega_i}_{(x,\xi_1,\ldots,\xi_{n-1})}=-\det(d\pi_i,d\pi_i,\xi_1-\xi_i,\ldots,\widehat{\xi_i-\xi_i},\ldots,\xi_{n-1}-\xi_i),
  \]
  where
  \begin{align*}
      \pi_i : E(\partial \B) & \to V \\
    (x,\xi_1,\ldots,\xi_{n-1}) & \mapsto \xi_i.
  \end{align*}
\end{definition}

Using these $2$-forms we are able to express $G^\ast \eta_{\B^\ast}$ as follows.

\begin{proposition}\label{prop:pullback}
  \begin{equation}\label{eq:main_pullback}
    G^*\eta_{\B^\ast} =\frac{(n-1)!}{L} \, \, \omega_1\wedge\cdots\wedge\omega_{n-1}.
  \end{equation}
\end{proposition}

The rest of this section is devoted to the proof of the proposition above.

%%%%%%%%%%%%%%%%%%%
\subsection{Technical lemmas}
%%%%%%%%%%%%%%%%%%%

Let us begin with the following computation.

\begin{lemma}\label{lem:pullback_omega0}
  \[
    (F^\ast G^\ast  \eta_{\B^\ast})_{(\lambda,x)}=\frac{(n-1)!}{\prod_{i=1}^{n-1} \lambda_i^2} \, \, d \lambda_1\wedge\cdots \wedge d \lambda_{n-1}\wedge \zeta^1(x)\wedge\cdots\wedge \zeta^{n-1}(x)
  \]
  where $\{\LL(x),\zeta^1(x),\ldots,\zeta^{n-1}(x)\}$ is the dual basis of $\{x,\zeta_1(x),\ldots,\zeta_{n-1}(x)\}$.
\end{lemma}
\begin{proof}
Note first that
\[
  F^\ast G^\ast \omega_{\B^\ast}=F^\ast H^\ast\pi^\ast\omega_{\B^\ast}=F^\ast H^\ast\omega_{V^*}
\]
where we have abusively denoted $i^\ast \omega_{V^\ast}$ by $\omega_{V^\ast}$ (here $i$ stands for the canonical inclusion $V^\ast \times \partial \B \hookrightarrow V^\ast \times V$).

Fixing $x_0\in U$ we write
\[
  \omega_{V^*}=x_0\wedge \LL(x_0)+\sum_{i=1}^{n-1} \zeta_i(x_0)\wedge \zeta^i(x_0)
\]
globally on $V^*\times V$.

Since $H\circ F(\lambda,x)=(p\circ F(\lambda,x),x)=(\sum_{i=1}^{n-1} \zeta^i(x)/\lambda_i,x)$  we have
\[
  d(H\circ F)_{(\lambda,x)}\left({\partial }/{\partial \lambda_i}\right)=\left(-\zeta^i(x)/{\lambda_i^2},0\right)\qquad \text{and} \qquad d(H\circ F)_{(\lambda,x)}(\zeta_i(x))=(\ast, \zeta_i(x)).
\]
Thus, modulo terms of the form $\zeta^i(x_0)\wedge\zeta^j(x_0)$,
\begin{align*}
  (F^*H^*\omega_{V^*})_{(\lambda,x_0)}
    & \equiv\sum_{i=1}^{n-1} \omega_{V^*}(d(H\circ F)_{(\lambda,x_0)}\left({\partial }/{\partial \lambda_i}\right),d(H\circ F)_{(\lambda,x_0)}(\zeta_i(x_0))) \, d\lambda_i\wedge \zeta^i(x_0) \\
    & = -\sum_{i=1}^{n-1} \frac{1}{\lambda_i^2} \, d\lambda_i\wedge \zeta^i(x_0).
\end{align*} The statement follows.
\end{proof}

Secondly we prove the following identity.

\begin{lemma}\label{lem:pullback_omega}
  \[
    F^\ast(\omega_1\wedge \cdots \wedge \omega_{n-1})_{(\lambda,x)}=  L(\zeta(x)) \, \, \left(\prod_{i=1}^{n-1} \lambda_i^{n-1}\right) \, \, d\lambda_1\wedge\cdots\wedge d\lambda_{n-1}\wedge \zeta^1(x)\wedge \cdots \wedge \zeta^{n-1}(x).
  \]
\end{lemma}

\begin{proof}
Fix $x_0\in U$ and put $\xi_i=\zeta_i(x_0)$. Considering $E(\partial \B)\subset V\times V^{n-1}$ we have
\begin{align*}
  dF_{(\lambda,x_0)}\left(\partial /\partial \lambda_j\right)
    & =(0 \,;\,0,\ldots,0,\xi_j,0,\ldots,0),\\
  \intertext{and}
  dF_{(\lambda,x_0)} (\xi_j)
    & =(\xi_j\,;\,\lambda_1 (d\zeta_1)_{x_0}(\xi_{j}) ,\ldots,\lambda_{n-1}(d\zeta_{n-1})_{x_0}(\xi_{j})).
\end{align*}
Thus
\[
  F^*\omega_i\left(\frac{\partial }{\partial \lambda_j},\frac{\partial }{\partial \lambda_k}\right)=0
\]
for any $i,j,k$ and
\[
  F^*\omega_i\left(\frac{\partial }{\partial \lambda_j},\xi_k\right)=0
\]
for any $i,j,k$ such that $i\neq j$, while
\begin{align*}
  F^*\omega_i\left(\frac{\partial }{\partial \lambda_i},\xi_j\right)
    & =-\det(\xi_i ,\lambda_i (d\zeta_i)_{x_0}(\xi_j),\lambda_1 \xi_1 - \lambda_i \xi_i ,\ldots, \widehat{\lambda_i \xi_i - \lambda_i \xi_i},\ldots, \lambda_{n-1} \xi_{n-1}  - \lambda_i \xi_i ) \\
    & =-\det(\xi_i ,\lambda_i (d\zeta_i)_{x_0}(\xi_j),\lambda_1 \xi_1 ,\ldots, \widehat{\lambda_i \xi_i },\ldots, \lambda_{n-1} \xi_{n-1} ) \\
    & =(-1)^{i+1}\, \, \left(\prod_{k=1}^{n-1} \lambda_k\right) \, \, \det(\nabla_{\xi_j}\zeta_i,\xi_1,\ldots, \xi_{n-1}),
\end{align*}
that is
\[
  F^*\omega_i\left(\frac{\partial }{\partial \lambda_i},\xi_j\right)=(-1)^{i+1}\, \, \left(\prod_{k=1}^{n-1} \lambda_k\right) \, \,  L_{j,i}(x_0,\xi_1,\ldots,\xi_{n-1}).
\]
Therefore, putting $L_{j,i}=L_{j,i}(x_0,\xi_1,\ldots,\xi_{n-1})$ and  $\xi^j=\zeta^j(x_0)$, we have
\begin{align*}
  &\hspace{-1em}F^\ast(\omega_1\wedge\cdots\wedge\omega_{n-1})_{(\lambda,x_0)} \\
  &=(-1)^{n-1}\, \, (-1)^{\frac{(n-1)n}{2}}\, \, \left(\prod_{i=1}^{n-1} \lambda_i^{n-1}\right) \, \,  \left(\sum_{j=1}^{n-1} L_{j,1} d\lambda_1\wedge \xi^j\right)\wedge\cdots\wedge\left(\sum_{j=1}^{n-1} L_{j,n-1} d\lambda_{n-1}\wedge \xi^j\right)\\
  &=\left(\prod_{i=1}^{n-1} \lambda_i^{n-1}\right) \, \,  d\lambda_1\wedge\cdots\wedge d\lambda_{n-1}\wedge\left(\sum_{j_1,\ldots,j_{n-1}} L_{j_1,1}\cdots L_{j_{n-1,n-1}}\xi^{j_1}\wedge\cdots \wedge \xi^{j_{n-1}}\right)\\
  &=\left(\prod_{k=1}^{n-1} \lambda_k^{n-1}\right) \, \,  \det (L_{i,j})_{i,j} \, \,  d\lambda_1\wedge\cdots\wedge d\lambda_{n-1}\wedge \xi^{1}\wedge\cdots \wedge \xi^{n-1}
\end{align*}
from which we get the announced formula.
\end{proof}

%%%%%%%%%%%%%%%%%%%
\subsection{Proof of Proposition \ref{prop:pullback}.}
%%%%%%%%%%%%%%%%%%%

Let us denote
 \[
  \Omega \coloneqq \frac{(n-1)!}{L} \, \, \omega_1\wedge\cdots\wedge\omega_{n-1}.
\]
The proposition will follow easily after proving that
\begin{enumerate}
  \item $F^\ast\Omega=F^\ast G^\ast \eta_{\B^\ast}$;
  \item $i_X\Omega=0$ for all $X \in \ker dG$.
\end{enumerate}
Indeed, given $\xi=(x,\xi_1,\ldots,\xi_{n-1})\in E(\partial \B)$ we can take a local section $\zeta$ of $E(\partial \B)$ such that $\zeta(x)=\xi$. In particular the map $F$ associated to this section satisfies that $F((1,\ldots,1),x)=\xi$. Then we observe that
\begin{equation*}\label{eq:direct_sum}
  \operatorname{im} dF_{((1,\ldots,1),x)} \oplus \ker d G_{\xi}=T_\xi E(\partial \B),
\end{equation*}
because $G\circ F$ is a diffeomorphism  according to Lemma \ref{lem:pullback_omega0}. Equality \eqref{eq:main_pullback} follows  directly using points  (1) and  (2).

Let us check point (1).
Note that (e.g. by Proposition \ref{prop:L_to_det})
\begin{equation}\label{eq:pullback_L}
   L\circ F(\lambda,x)= \left(\prod_{i=1}^{n-1} \lambda_i^{n+1}\right) \, \,  L(x,\zeta_1(x),\ldots,\zeta_{n-1}(x)).
\end{equation}
Hence, by Lemma \ref{lem:pullback_omega}, at $\xi=F(\lambda,x)$ we have
\begin{align*}
  F^\ast \Omega & =\frac{(n-1)!}{L\circ F(\lambda,x)}  \, \,F^\ast (\omega_1\wedge \cdots \wedge \omega_{n-1}) \\
                & =\frac{(n-1)!}{\prod_{i=1}^{n-1} \lambda_i^2}  \, \, d\lambda_1\wedge\cdots\wedge d\lambda_{n-1}\wedge \zeta^1(x)\wedge \cdots \wedge \zeta^{n-1}(x),
\end{align*}
which is precisely $F^\ast G^\ast\eta_{\B^\ast}$ by Lemma \ref{lem:pullback_omega0}.

In order to check point  (2), we easily verify that
\begin{equation}\label{eq:ker_dG}
  \ker dG_{\xi}=\langle (0,Z_{ij})\in V\times (V)^{n-1} \mid i,j \in \{1,\ldots,n-1\} \, \, \text{with} \, \, i\neq j \rangle
\end{equation}
where
\[
  Z_{ij}=(0,\ldots,0,\underbrace{\xi_j-\xi_i}_{i\text{th position}},0,\ldots,0)\in V.
\]
Indeed we have $G(x,\xi_1,\ldots,\xi_{i-1},(1-t)\xi_i+t\xi_j,\xi_{i+1},\ldots,\xi_{n-1})=G(x,\xi_1,\ldots,\xi_{n-1})$ for all $t$ and differentiating the curve
\[
  t \mapsto (\xi_1,\ldots,(1-t)\xi_i+t\xi_j,\ldots,\xi_{n-1})
\]
at time $0$ gives the vector $Z_{ij}$. The linear independence of $(\xi_1,\ldots,\xi_{n-1})$ then ensures that the family $\{Z_{ij} \mid i,j \in \{1,\ldots,n-1\} \, \, \text{with} \, \, i\neq j\}$ is also linearly independent. Its cardinality is $(n-1)(n-2)$, which coincides with $\dim\ker dG=\dim E(\partial \B)-2(n-1) $ as $G\circ F$ is a diffeomorphism. This shows \eqref{eq:ker_dG}.

Next we see that if $i\neq k$, then
\[
  (i_{(0,Z_{ij})}\omega_k)_{\xi}=-\det(d\pi_k(Z_{ij}),d\pi_k,\xi_1-\xi_k,\ldots,\widehat{\xi_k-\xi_k},\ldots,\xi_{n-1}-\xi_k) =0
\]
and that
\begin{align*}
  (i_{(0,Z_{ij})}\omega_i)_{\xi}
    & =-\det(d\pi_i(Z_{ij}),d\pi_i,\xi_1-\xi_i,\ldots,\widehat{\xi_i-\xi_i},\ldots,\xi_{n-1}-\xi_i)\\
    & =-\det(\xi_j-\xi_i,d\pi_i,\xi_1-\xi_i,\ldots,\widehat{\xi_i-\xi_i},\ldots,\xi_{n-1}-\xi_i)\\
    & =0,
\end{align*}
which proves point (2). $\hfill \square$

%%%%%%%%%%%%%%%%%%%%%%%%%%%%%%%%%%%%%
\section{First variation of the dual area under translations}\label{sec:firstvariation}
%%%%%%%%%%%%%%%%%%%%%%%%%%%%%%%%%%%%%

Recall that we have chosen some non-trivial alternate multilinear $n$-form denoted by $\det$ on the $n$-dimensional vector space $V$. Recall also that $\B\subset V$ is a convex body containing the origin in its interior and has smooth and positively curved boundary (i.e. the normed vector space $(V,\|\cdot\|_{\B})$ is Minkowski). Let $\K$ denote another convex body of $V$.

In this section we prove Theorem \ref{th:firstvariation} by explicitly computing
\[
  \left.\frac{d}{dt}\right|_{t=0}\area_{\B^\ast}({\partial(\K-tv)^\ast})
\]
for any $v\in V$.

%%%%%%%%%%%%%%%%%%%%%%%%%%%%%%%%%%%%%
\subsection{Preliminaries}
%%%%%%%%%%%%%%%%%%%%%%%%%%%%%%%%%%%%%

Let us begin with the following simple fact.

\begin{lemma}\label{lem:affine}
 Let $\xi=(x,\xi_1,\ldots,\xi_{n-1})\in E(\partial B)$ and $y\in V\setminus \{0\}$. The affine line $G(\xi)$ is contained in the affine hyperplane $H_y=\{p\in V^* \colon \langle p,y\rangle=1\}$ if and only if $y$ belongs to $\mathrm{aff}(\xi)$, the affine hull of $\xi_1,\ldots, \xi_{n-1}$.
\end{lemma}
\begin{proof}
Recall that 
\[
 G(\xi)= \{p\in V^*\colon \langle\xi_1, p\rangle=\cdots=\langle\xi_{n-1},p\rangle=1\}=q+\langle \LL(x)\rangle
\]
for some $q\in V^*$. Hence $G(\xi)\subset H_y$ if and only if $y\in \ker \LL(x)=\vect{\xi_1,\ldots,\xi_{n-1}}$ and $\langle q,y\rangle =1$. This is equivalent to $y=\sum_i\lambda_{i}\xi_i$ with $\sum_i\lambda_i=\langle q,y\rangle=1$, which means that $y$ is in the affine hull of the $\xi_i$.
\end{proof}

Given $v\in V$, let $E_v$ denote the set of $\xi\in E(\partial B)$ such that $-v\notin \mathrm{aff}(\xi)$. Let  $\overline T_v\colon E_v\to E(\partial \B)$ be given by
\[
  \overline T_v(x,\xi_1,\ldots, \xi_{n-1})= (x_v,\xi_1+v,\ldots,\xi_{n-1}+v).
\]
Here $x_v\in \partial \B$ is uniquely determined by the conditions that $\xi_i+v\in T_{x_v}\partial \B$ for all $i$  and that the basis $\{x_v,\xi_1+v,\ldots, \xi_{n-1}+v\}$ is negatively oriented. Note that $-v\notin\mathrm{aff}(\xi)$ ensures that $\xi_1+v,\ldots,\xi_{n-1}+v$ are linearly independent.

Let $T_v\colon G(E_v)\rightarrow G_+(V^*)$ be defined by
\[
  T_v(G(\xi))=G(\overline T_v(\xi)).
\]
To prove that this is independent of $\xi$, assume $G(\xi)=G(\xi')$ for some $\xi'=(x',\xi_1',\ldots,\xi_{n-1}')$. By Lemma \ref{lem:affine}, each $\xi_i'$ is in the affine hull of $\xi_1,\ldots,\xi_{n-1}$.  
Since $\xi_i'+v$ is clearly also in the affine hull of the $\xi_j+v$, we get by Lemma \ref{lem:affine} that $G(\overline T_v(\xi))=G(\overline T_v(\xi'))$ up to orientation. 
Since $v\mapsto x_{v}$ is continuous, the orientations of $G(\overline T_v(\xi))$ and $G(\overline T_v(\xi'))$ must  agree or disagree for all $v$; and they coincide for $v=0$. This shows that $T_v$ is well-defined.

\begin{lemma}\label{lem:lines_translation}
  For $v\in V$ and $L\in G_+(V^*)$ we have
  \begin{equation}\label{eq:equivalence}
    L\cap (\K-v)^\ast \neq \emptyset\quad \Longleftrightarrow\quad L\in G(E_v) \mbox{ and } T_v(L)\cap \K^\ast \neq\emptyset.
  \end{equation}
\end{lemma}
\begin{proof}
Given $\zeta=(x,\zeta_1,\ldots,\zeta_{n-1})\in E(\partial B)$  and a convex body $Q$ in $V$, note that $G(\zeta)\cap Q^\ast\neq \emptyset$ if and only if every hyperplane containing $G(\zeta)$ intersects $Q^*$. By Lemma \ref{lem:affine} this is equivalent to the fact that $H_y\cap Q^\ast\neq\emptyset$ for every $y$ in the affine hull of $\zeta_1,\ldots,\zeta_{n-1}$. In turn, this is equivalent to $y\notin Q$ for all such $y$.

Let now $L=G(\xi)$. Taking $\zeta=\xi$ and $Q=\K-v$ we get that the left hand side of \eqref{eq:equivalence} holds if and only if $y\notin \K-v$ for every $y\in\mathrm{aff}(\xi)$. Since $0\in\inter{\K}$, this is equivalent to the fact that $-v\notin \mathrm{aff}(\xi)$ and $z\notin K$ for all $z\in\mathrm{aff}(\overline T_v(\xi))$. 
Taking $\zeta=\overline T_v(\xi)$ and $Q=\K$ this condition is equivalent to the right hand side of \eqref{eq:equivalence}.
\end{proof}
 
Note that, by the first part of the previous proof, if $G(\xi)\cap K^*\neq\emptyset$ and $v\in K$,  then $\xi\in E_{-v}$.  Then, by Proposition \ref{prop:AP} and Lemma \ref{lem:lines_translation}, for  $tv\in\inter{K}$ 
\begin{align*}
  \area_{\B^\ast}(\partial (\K-tv)^\ast) 
  & =\frac{1}{(n-1)!\eps_{n-1}}\int_{\{L\in G(E_{tv})\colon T_{tv}(L)\cap {\K^\ast}\neq \emptyset\}} \eta_{\B^\ast} \\
    & =\frac{1}{(n-1)!\eps_{n-1}}\int_{\{T_{-tv}(L)\colon L\in G_+(V^*), L\cap {\K^\ast}\neq \emptyset\}} \eta_{\B^\ast} \\
    & =\frac{1}{(n-1)!\eps_{n-1}}\int_{\{L\in G_+(V^*)\colon L\cap {\K^\ast}\neq \emptyset\}}T_{-tv}^\ast \, \eta_{\B^\ast}.
\end{align*}

Consider the vector field $Z_v$ on $E(\partial \B)$ given by
\[
 Z_v(x,\xi_1,\ldots,\xi_{n-1})=\left.\frac{d}{dt}\right |_{t=0} \overline T_{tv}(x,\xi_1,\ldots,\xi_{n-1}),
\]
and recall that
\[
  G^*\eta_{\B^\ast} = \Omega  =  \frac{(n-1)!}{L} \, \, \omega_1\wedge\cdots\wedge\omega_{n-1}.
\]
Let now  $\{U_i\}$ be a finite cover of $\partial \B$, take a local section $\zeta_i$ of each $E(U_i)$, and let $F_i\colon (\R\setminus\{0\})^{n-1}\times U_i\to E(\partial \B)$ be the corresponding map. Since $G\circ F_i$ is a diffeomorphism onto its image, taking a partition of unity $\rho_i$ subordinate to $\{U_i\}$, we have
\begin{align*}
  \area_{\B^\ast}(\partial (\K-tv)^\ast)
    & =\frac{1}{(n-1)!\eps_{n-1}}\sum_i\int_{(G\circ F_i)^{-1}(\{L\colon L\cap {\K^\ast}\neq \emptyset\})}\rho_i\cdot F_i^\ast G^\ast T_{-tv}^\ast \, \eta_{\B^\ast} \\
    & =\frac{1}{(n-1)!\eps_{n-1}}\sum_i\int_{(G\circ F_i)^{-1}(\{L\colon L\cap {\K^\ast}\neq \emptyset\})}\rho_i\cdot F_i^\ast \overline T_{-tv}^\ast \, \Omega
\end{align*}
and we deduce that
\begin{equation}\label{eq:variation_lie}
  \left.\frac{d}{dt}\right |_{t=0}  \area_{\B^\ast}(\partial(\K-tv)^\ast)=-\frac{1}{(n-1)!\eps_{n-1}}\sum_i\int_{(G\circ F_i)^{-1}(\{L\colon L\cap {\K^\ast}\neq \emptyset\})}\rho_i\cdot F_i^\ast  \LL_{Z_v}\Omega.
\end{equation}

\begin{proposition}\label{prop:lie_derivative}
For any $\xi=(x,\xi_1,\ldots,\xi_{n-1})\in E(\partial \B)$ and any $v \in V$
\[
  \LL_{Z_v}\Omega= -(n-1)!\, \,\frac{\LL_{Z_v} L}{L^2} \, \,\omega_1\wedge\cdots\wedge\omega_{n-1}
\]
and
\[
  \LL_{Z_v} L (\xi)=(n+1)\, \,(\det(\Xi(x),\xi_1,\ldots, \xi_{n-1}))^{n} \, \, \left(\sum_{i=1}^{n-1}{\det(\Xi(x),\xi_1,\ldots,\xi_{i-1},v,\xi_{i+1},\ldots, \xi_{n-1})}\right)
\]
where $\Xi$ is the interior equiaffine normal field of $\partial \B$.
\end{proposition}

\begin{proof}
The first equality follows from $\LL_{Z_v} \omega_i=0$. As for the second, by Proposition \ref{prop:L_to_det} we have
\[
  L(\overline T_v(\xi))=L(x_{tv},\xi_1+tv,\ldots, \xi_{n-1}+tv)=\det(\Xi(x_{tv}),\xi_1+tv,\ldots, \xi_{n-1}+tv)^{n+1}.
\]
Hence
\[
  \LL_{Z_v}L(\xi) = (n+1) \, \, \det(\Xi(x),\xi_1,\ldots, \xi_{n-1})^n \, \, \left.\frac{d}{dt}\right |_{t=0} \det(\Xi(x_{tv}),\xi_1+t v,\ldots ,\xi_{n-1}+tv).
\]
The proposition follows by noting that $\left.\frac{d}{dt}\right |_{t=0} \Xi(x_{tv})=\nabla_X \Xi(x)$ with $X=\left.\frac{d}{dt}\right |_{t=0} x_{tv}$ which according to the first identity of (\ref{eq:conditions}) implies that
\[
  \left.\frac{d}{dt}\right |_{t=0} \det(\Xi(x_{tv}),\xi_1+t v,\ldots ,\xi_{n-1}+tv)=\sum_{i=1}^{n-1}\det\left(\Xi(x),\xi_1,\ldots ,\xi_{i-1},v,\xi_{i+1},\ldots,\xi_{n-1}\right).
\]
\end{proof}

%%%%%%%%%%%%%%%%%%%%%%%%%%%%%%%%%%%%%
\subsection{Proof of Theorem \ref{th:firstvariation}}
%%%%%%%%%%%%%%%%%%%%%%%%%%%%%%%%%%%%%

Take a local section $\zeta\colon U\to E(\partial \B)$ mapping each $x\in U$ to a positive  orthonormal basis $\{e_1(x),\ldots,e_{n-1}(x)\}$ of $T_x\partial \B$ with respect to the Riemannian metric $g$. For all $x \in U $ let $\{e^1(x),\ldots,e^{n-1}(x),\Xi^*(x)\}$ denote the dual basis of $\{e_1(x),\ldots,e_{n-1}(x),\Xi(x)\}$. Since $\ker \Xi^*(x)=\vect{e_1(x),\ldots,e_{n-1}(x)}$, we have $\Xi^*(x)\in\langle \LL(x)\rangle$.  Take the smooth measure $\mu=e^1\wedge\cdots\wedge e^{n-1}$ on $U$.  For each $x\in U$ the dual measure $\nu_x$ on $\ker\Xi(x)\simeq T^*_x\partial \B$ of the Lebesgue measure $\mu_x$ is given by $\nu_{x}=e_1(x)\wedge\cdots\wedge e_{n-1}(x)$.

First observe that if $F$ denotes the map corresponding to $\zeta$ we have
\begin{align*}
  L\circ F(\lambda,x)
    & =L(x,\lambda_1e_1(x),\ldots,\lambda_{n-1}e_{n-1}(x))                                               \\
    & =\left(\prod_{i=1}^{n-1} \lambda_i^{n+1}\right) \, \, \det(\Xi(x),e_1(x),\ldots, e_{n-1}(x))^{n+1} \\
    & =\left(\prod_{i=1}^{n-1} \lambda_i^{n+1}\right)
\end{align*}
using the fact that  $\{e_1(x),\ldots,e_{n-1}(x)\}$ is a positive orthonormal basis together with point ii) in Theorem \ref{thm:equiaffine}, and that similarly
\[
  \LL_{Z_v} L\circ F(\lambda,x)=(n+1)\, \,\left(\prod_{i=1}^{n-1} \lambda_i^{n+1}\right)  \, \, \sum_{i=1}^{n-1} \frac1{\lambda_i} \langle e^i(x),v\rangle
\]
where $\langle\cdot,\cdot\rangle$ denotes the canonical duality pairing.
Furthermore by Lemma \ref{lem:pullback_omega} the following holds:
\[
  F^\ast(\omega_1\wedge \cdots \wedge \omega_{n-1})_{(\lambda,x)}=  \left(\prod_{i=1}^{n-1} \lambda_i^{n-1}\right) \, \, d\lambda_1\wedge\cdots\wedge d\lambda_{n-1}\wedge e^1(x)\wedge \cdots \wedge e^{n-1}(x).
\]
Therefore, using Proposition \ref{prop:lie_derivative},  we get the following formula:
\[
  F^*(\LL_{Z_v}\Omega)_{(\lambda,x)}=-\frac{(n+1)!}{n} \, \,  \left\langle  \sum_{i=1}^{n-1} \frac1{\lambda_i}   e^i(x),v\right\rangle  \, \, \left(\prod_{i=1}^{n-1}  \frac{1}{\lambda_i^{2}}\right)  \, \,  d\lambda_1\wedge\cdots\wedge d\lambda_{n-1}\wedge d\mu(x).
\]
Let us define the map
\begin{align*}
    p_{x}: (\R\setminus\{0\})^{n-1} & \to \ker \Xi(x) \subset V^\ast \\
  (\lambda_1,\ldots, \lambda_{n-1}) & \mapsto \sum_{i=1}^{n-1} \frac1{\lambda_i}   e^i(x)
\end{align*}
which is a diffeomorphism onto its image that satisfies
\[
 G\circ F( \lambda,x)=p_{x}(\lambda)+\langle \LL(x)\rangle_+.
\]
In particular, we have
\begin{align*}
  (G\circ F)^{-1}(\{L\colon L\cap {\K^\ast}\neq \emptyset\})
    & =\{(\lambda,x)\in (\R\setminus\{0\})^{n-1}\times U\colon G\circ F(\lambda,x)\cap {\K^\ast}\neq \emptyset\} \\
    & =\{(\lambda,x)\in (\R\setminus\{0\})^{n-1}\times U\colon p_{x}(\lambda)\in\pi_{x}({\K^\ast})\}
\end{align*} where $\pi_x\colon V^*\to \ker\Xi(x)$ is the linear projection with $\ker\pi_x=\langle \LL(x)\rangle$.

A simple computation shows that
\begin{align*}
(p^{-1}_{x})_* \, \nu_{x}&=\left(\prod_{i=1}^{n-1}  \frac{1}{\lambda_i^{2}}\right)  \, \,  d\lambda_1\wedge\cdots\wedge d\lambda_{n-1}
\end{align*}
as  measures. Therefore
\begin{align*}
 & \hspace{-2em} \int_{(G\circ F)^{-1}(\{L\colon L\cap {\K^\ast}\neq \emptyset\})} F^*\LL_{Z_v}\Omega \\
 & =-\frac{(n+1)!}{n}\int_{x\in U}\left(\int_{\lambda\in p_{x}^{-1} (\pi_{x}(\K^\ast))} \langle p_{x}(\lambda),v\rangle \, \, d(p^{-1}_{x})_*(\nu_{x})\right)  \,  d \mu(x) \\
 & =-\frac{(n+1)!}{n}\int_{x\in U}\left(\int_{q\in \pi_{x}(\K^\ast)} \, \, \langle q,v\rangle \,  d\nu_{x}(q)\right)  \,  d\mu(x)                                           \\
 & =-\frac{(n+1)!}{n}\left\langle\int_{x\in U}  \left(\int_{\pi_{x}(\K^\ast)}  \, \, q  \, d\nu_{x}(q) \right)  \, d\mu(x),v\right\rangle.
\end{align*}
Covering $\partial \B$ with local sections and using \eqref{eq:variation_lie}, we deduce the stated formula:
\[
  \left.\frac{d}{dt}\right |_{t=0}  \area_{\B^\ast}(\partial(\K-tv)^\ast)=\frac{n+1}{\eps_{n-1}}\left\langle\int_{x\in \partial B}  \left(\int_{\pi_{x}(\K^\ast)}  \, \, q  \, d\nu_{x}(q) \right)  \, d\mu(x),v\right\rangle. \qquad \qquad  \square
\]

%%%%%%%%%%%%%%% %%%%%%%%%%%%%%%
%BIBLIOGRAPHY
%%%%%%%%%%%%%%%%%%%%%%%%%%%%%%

\end{document}